\documentclass{article}

\usepackage[utf8]{inputenc}
\usepackage{microtype}

\usepackage{fullpage}
\usepackage{amsmath,amssymb,amsthm,mathtools}
\mathtoolsset{showonlyrefs=false}

\usepackage{bm}
\usepackage{mathrsfs}
\usepackage{bbm}

\usepackage{pgfplots}
\pgfplotsset{compat=1.18}
\usepackage{tikz}

\theoremstyle{plain}
\newtheorem{theorem}{Theorem}[section]
\newtheorem{lemma}[theorem]{Lemma}
\newtheorem{proposition}[theorem]{Proposition}
\newtheorem{corollary}[theorem]{Corollary}

\theoremstyle{definition}
\newtheorem{definition}[theorem]{Definition}
\newtheorem{example}[theorem]{Example}

\theoremstyle{remark}
\newtheorem{remark}[theorem]{Remark}

\DeclareMathOperator{\dom}{dom}
\DeclareMathOperator{\cl}{cl}
\DeclareMathOperator{\an}{an}
\DeclareMathOperator{\interior}{int}

\DeclareMathOperator{\depth}{depth}
\DeclareMathOperator{\width}{width}

\DeclareMathOperator{\size}{size}
\DeclareMathOperator{\fr}{fr}

\DeclareMathOperator{\esssup}{ess sup}
\DeclareMathOperator{\argmin}{argmin}
\DeclarePairedDelimiter\abs{\lvert}{\rvert}
\DeclarePairedDelimiter\norm{\lVert}{\rVert}

\newcommand{\R}{\mathbb{R}}

\newcommand{\N}{\mathbb{N}}
\newcommand{\eps}{\varepsilon}

\usepackage{graphicx}
\usepackage{caption}
\usepackage{subcaption}

\usepackage[backend=biber,style=numeric,sorting=nyt,maxbibnames=10]{biblatex}
\usepackage{hyperref}
\hypersetup{
  pdftitle={Fast approximation and learning of binary classification tasks in o-minimal structures using ReLU neural networks},
  pdfauthor={Clemens Kinn, Philipp Petersen},
  pdfsubject={Machine Learning and O-minimality},
  pdfkeywords={Machine Learning, O-minimality},
  colorlinks=true,
  citecolor=black,
  linkcolor=black,
  urlcolor=black,
  pdfstartview={FitH}
}

\usepackage{enumitem}
\usepackage{todonotes}
\usepackage{xcolor}

\title{Fast approximation and learning of binary classification tasks in o-minimal structures using ReLU neural networks}

\author{
    Clemens Kinn\footnotemark[1]
\and Philipp Petersen\footnotemark[1]
}

\begin{document}

\maketitle
\begin{abstract}
We study binary classification tasks whose decision sets are given by definable sets in o-minimal expansions of the real field. Motivated by the cell decomposition theory of definable sets, we introduce traceable sets that isolate the analytical properties of definable sets used for fast approximation by ReLU neural networks. Under uniform bounds on the number of cells and a suitable extension assumption for the boundary functions, we prove that classifiers can be approximated in $L^p$ up to accuracy $\eps>0$ by ReLU neural networks of size $\mathcal{O}(\eps^{-p(n-1)/m})$, with depth independent of $\eps$ and polynomially bounded weights. This establishes quantitative approximation rates for suitable definable collections in o-minimal structures using ReLU neural networks. Furthermore, the stated approximation rates are also derived for definable maps under similar geometric assumptions on their smoothness domains.
We then obtain statistical learning rates for empirical risk minimization with hinge loss over the approximating ReLU neural network classes. For $N$ uniformly distributed samples, the predicted classifiers achieve expected misclassification error of order $N^{-m/(m+n-1)}$ up to an arbitrarily small polynomial loss.
\end{abstract}



\renewcommand{\thefootnote}{\fnsymbol{footnote}}
\footnotetext[1]{
\raggedright
Faculty of Mathematics, University of Vienna, 1090 Vienna, Austria.\\
\href{mailto:clemens.kinn@univie.ac.at}{\nolinkurl{clemens.kinn@univie.ac.at}};
\href{mailto:philipp.petersen@univie.ac.at}{\nolinkurl{philipp.petersen@univie.ac.at}}
}

\section{Introduction}\label{sec:intro}

Deep neural networks have become one of the most successful paradigms in modern data analysis, outperforming traditional methods in a wide range of tasks, from image recognition and speech processing to natural language understanding \cite{lecun2015deep}. In particular, machine learning can be applied successfully to study binary classification tasks that arise in areas such as image classification.
We consider noiseless binary classification on the compact box $[-1/2,1/2]^n$ with $n\in\N^*$, where each data point is deterministically part of one of two classes $A \subseteq [-1/2,1/2]^n$ or $A^c := [-1/2,1/2]^n\setminus A$. The classifier $\chi_A:[-1/2,1/2]^n \to \{0,1\}$ assigns to each data vector $v\in [-1/2,1/2]^n$ its label, $1$ if $v\in A$ and $0$ if $v\in A^c$. The set of discontinuities of $\chi_A$ is contained in the decision boundary $\cl(A) \cap \cl(A^c)$, whose regularity determines the hardness of the decision task.

Classification tasks are ubiquitous in applications, and their theoretical study using neural networks has led to work connecting approximation theory under geometric regularity of decision boundaries \cite{yarotsky2017errorboundsapproximationsdeep,petersen2018optimal} and statistical learning theory under different assumptions (e.g. probabilistic labelling, margin condition, noise) \cite{kim2021fast,petersen2021optimal,caragea2023neural,meyer2023optimal,garcia2025high}. We study the case of assigning deterministic labels.

Any quantitative result on approximation or learnability relies on a structural assumption on the target set $A$, for example smoothness of boundary functions \cite{petersen2018optimal}. Boundaries given by Barron-type functions are well-adapted to approximation and learning by ReLU neural networks \cite{caragea2023neural,petersen2021optimal,garcia2025high}. In many cases, optimal approximation and estimation rates are known. In general, the larger the hypothesis class becomes, the weaker the achievable rates.

In this paper, we study a framework for classification tasks based on tame geometry and mathematical model theory. More precisely, we extract the analytical properties of definable sets for any o-minimal expansion of the real field for fast approximation and learning via ReLU neural networks. Definable sets generalize the notion of semi-algebraic sets while maintaining strong structural results including cell decomposition theorems and a well behaved notion of dimension \cite{van1998tame}.

In particular, hypothesis classes for decision tasks coming from definable collections admit uniform finiteness properties such as a uniform bound on the number of connected components of each hypothesis and finite VC dimension \cite{laskowski1992vapnik}, making them a natural candidate class for learning-theoretic analysis.

For classical approximation, Barron function spaces are well suited for shallow networks, and flow-induced function spaces are used for deep residual networks \cite{ma2022barron}. 
A key feature of o-minimal structures is that the composition of two definable maps is again definable and
deep neural networks are built by repeatedly composing affine maps with non-linear activation functions. Consequently, many neural network architectures fit naturally into the unifying o-minimality framework \cite{kratsios2026feedforwardneuralnetworkdefinable}.

\subsection{Our contribution}

For every $m\in\N^*$, each definable set can be partitioned into \emph{$C^m$-cells} and for any definable map, its domain can be partitioned into $C^m$-cells on each of which its restriction is $C^m$. Inspired by the cell decomposition theory in o-minimal structures, we introduce the class of \emph{traceable sets} that generalize the structure of definable sets by dropping the definability requirement. To get fast approximation rates, we assume that the boundary functions of traceable sets satisfy compatibility conditions for the existence of global $C^m$ extensions. This leads to the consideration of \emph{extendable} traceable sets.

Our first main result, Theorem \ref{thm:mainapprox}, establishes quantitative ReLU approximation for characteristic functions of extendable traceable subsets of $[-1/2,1/2]^n$ for $n\geq 2$ with a uniform bound on the number of cylindrical pieces that additionally admit uniform $C^m$ extension bounds on a compact domain. For any $0<\eps<1/2$ and $1\leq p < \infty$, an $L^p$ approximation error of less than $\eps$ is achieved by a ReLU neural network of size $\mathcal{O}(\eps^{-p(n-1)/m})$, with depth bounded independently of $\eps$ and weights growing at most polynomially in $1/\eps$.
Building upon Theorem \ref{thm:mainapprox}, we approximate finite sums $f=\sum_j g_j\chi_{A^j}$ on $[-1/2,1/2]^n$, where the $A^j$ are disjoint extendable traceable sets and each $g_j$ has an extension with uniformly bounded derivatives through order $\lceil mn/(p(n-1))\rceil$. The same asymptotic $L^p$ approximation rate is achieved. In particular, this applies to definable maps admitting a $C^m$ cell decomposition whose cells are extendable as traceable sets. The construction builds on the ReLU approximation of horizon functions developed in \cite{petersen2018optimal}.

Our second main result, Theorem \ref{thm:statlearn}, concerns learning the decision region from independent uniformly chosen random points in $[-1/2,1/2]^n$. We derive learning rates via empirical risk minimization over the ReLU neural network classes that approximate the characteristic function of a decision region in $L^p([-1/2,1/2]^n)$. For $N\geq 1$ independent samples drawn from the Lebesgue distribution on $[-1/2,1/2]^n$, we obtain an upper bound of order $\mathcal{O}(N^{-m/(m+n-1)})$ on the minimax error, up to an arbitrarily small polynomial loss. The learning result is inspired by \cite[Theorem 5.7]{petersen2021optimal}.

\subsection{Related work}

An early conceptual bridge between machine learning and model theory was established by the observation that definable collections in model-theoretic structures with the non-independence property (NIP) have finite VC dimension \cite{laskowski1992vapnik}. 
This result in particular applies to o-minimal structures and it was shown that for o-minimal expansions of the real field, definable collections are PAC learnable under suitable measure-theoretic assumptions \cite{krapp2025measurabilityfundamentaltheoremstatistical}. Furthermore, stability theory has been used to study online learning, identifying Littlestone dimension and Shelah-2 rank \cite{chase2019model}. These developments suggest that many combinatorial complexity measures in machine learning could be interpreted through model-theoretic regularity properties.

Model theory also appeared in the study of the complexity of neural network hypothesis classes. For example, polynomial bounds on the VC dimension of sigmoidal and Pfaffian neural networks are known \cite{karpinski1997polynomial}, providing an explicit example where definability in an o-minimal structure yields quantitative complexity estimates for classes of neural networks. Recently, these results have been revisited in the modern setting of machine learning for many other model classes, showing definability and PAC learnability of most practically used neural network architectures \cite{kratsios2026feedforwardneuralnetworkdefinable}.

On the geometric side, o-minimal structures provide strong structural results for definable sets and definable maps. These include classical cell decomposition \cite{van1998tame}, refined Lipschitz decompositions of definable sets \cite{pawlucki2008lipschitz}, and stratifications of definable sets \cite{van1998tame,le1998verdier,fischer2007minimal}. For example, stratifications of functions were used for algorithmic approximation of definable maps by piecewise polynomial maps \cite{bareilles2026piecewise}.

Furthermore, o-minimality was proposed as a rigorous setting for studying aspects of deep learning such as non-smooth, non-convex optimization problems arising in training procedures of neural networks; for a survey, see \cite{bareilles2025deep}.

\section{Notation and conventions...}

\emph{Throughout this paper, $n$ ranges over $\N:=\{0,1,\ldots\}$, $m$ ranges over $\N^*:= \N \setminus \{0\}$, $p$ ranges over the real interval $[1,\infty)$, and $Q_n:=[-1/2,1/2]^n$}. Further, let $[n] := \{1,\ldots,n\}$ where $[0] := \emptyset$. For a real number $r \geq 0$, denote its integer part by $\lfloor r \rfloor$. The letters $i,j,k,l,\ell$ are reserved for integers unless explicitly stated otherwise.

\subsection{... concerning sets and maps}

Suppose $A, B$, and $C$ are sets with $C \subseteq A$.
Denote the graph of a map $f:A \to B$ by $\Gamma(f) \subseteq A \times B$ and the restriction of $f$ to $C$ by $f|_C$. For $A=\emptyset$, the map with the empty graph is the empty map $A \to B$.
Further, let $\chi_C:A\to\{0,1\}$ be given by $\chi_C(x):=1$ if $x \in C$, and $\chi_C(x):=0$ if $x \in A\setminus C$. 
The power set of $A$ is denoted as $\mathcal{P}(A)$.

We consider $\R^n$ as the usual Euclidean vector space with standard basis $\{e_i\}_{i \in [n]}$  and use the convention that $\R^0 := \{*\}$ is a singleton set. Any map $f:\{*\}\to A$ is identified with its value $f(*) \in A$.
Further, the open and closed balls of radius $\eps>0$ around $x \in \R^n$ are denoted by $B_\eps(x)$ and $\overline{B}_\eps(x)$. 
For $A \subseteq \R^n$, we let $\cl(A)$ be the closure of $A$ and distinguish between the boundary $\partial A := \cl(A) \setminus \interior(A)$ of $A$ and the frontier $\fr(A) := \cl(A) \setminus A$ of $A$. 

We let $\R_\infty := \R \cup \{\pm \infty\}$ be the ordered set extending $\R$ by $-\infty < \R < \infty$ and view $\R_\infty$ as a topological space equipped with the order topology. The operations $\sup, \inf, \min,$ and $\max$ take values in $\R_\infty$ and equal zero over empty index sets.

\subsection{... concerning projections and fibres}

Let $i\in[n]$. Denote the projection of $\R^n$ onto the first $i$ coordinates by $\pi^{\leq i}_n: \R^n \to \R^i$, $(x_1,\ldots,x_n) \mapsto (x_1,\ldots,x_i)$. We also use $\pi_n^{\leq 0}:\R^n \to \R^0$, $\pi_n^{< i} := \pi_n^{\leq i-1}$ and $\pi_{n+1} := \pi_{n+1}^{\leq n}$. 
For the fibre of a set $A \subseteq \R^n$ above $x=(x_1,\ldots,x_i) \in \R^i$, we write $A(x) := \{y \in \R^{n-i}: (x,y) \in A\} \subseteq \R^{n-i}$. So the fibre of $A$ above $x$ is non-empty if and only if $x \in \pi_n^{\leq i}(A)$. 
Similarly, for a map $f:A \to \R$, we let $f(x): A(x) \to \R$, $y \mapsto f(x,y)$ be the map given by $f(x)(y)=f(x,y)$. This notation causes no confusion since $x \notin f^{-1}(\R)$ for $i<n$ and if $x\in A$, then $f(x)$ is consistent with its usual meaning using the convention about maps $\R^0 \to \R$, or $A(x) = \emptyset$ and $f(x)$ is the empty map. Also, set $A():=A$ and $f():=f$ for the empty tuple.

\subsection{... concerning smoothness and measurability}\label{sec:smoothandmeasure}

For a $C^m$ function $f:U \to \R$, $U \subseteq \R^n$ open, $\alpha \in \N^n$, $\abs{\alpha} := \alpha_1 + \cdots + \alpha_n \leq m$, let $\partial^\alpha f: U \to \R$ be the continuous $\alpha$-th derivative of $f$. In case $\alpha = e_i$, we also use $\partial^i f$ instead of $\partial^{e_i} f$.
More generally, a function $g:A \to \R$ for arbitrary $A \subseteq \R^n$ is $C^m$ if there is some open $U \subseteq \R^n$ with $A \subseteq U$ and a $C^m$ map $f:U \to \R$ with $g=f|_A$. Any function $f:\R^0 \to \R$ is $C^m$ using $\partial^\alpha f = f$ for $\alpha=()$. Per convention, the empty map is $C^m$. The set $C^m(A)$ denotes the set of all $C^m$ maps $A \to \R$ and $C(A)$ is the class of continuous maps $A\to\R$.

For a Lebesgue measurable function $f:\Omega \to \R$ on a Lebesgue measurable set $\Omega \subseteq \R^n$, write
$$\norm{f}_{L^p(\Omega)} := \left(\int_\Omega \abs{f}^p d\lambda^n\right)^{1/p}\in\R_\infty$$
where $\lambda^n$ is the Lebesgue measure on $\R^n$. By convention, $\lambda^0(\{*\})=1$.
Also, write $\norm{f}_{L^\infty(\Omega)} := \esssup_{x\in \Omega} \abs{f(x)} \in \R_\infty$.
If $f$ is $C^m$ on $\Omega$ and $F:U\to\R$, $U \subseteq \R^n$ open, is a $C^m$ extension of $f$ where $\Omega \subseteq U$, we use $\norm{F}_{W^{m,\infty}(U)} :=\max_{\abs{\alpha} \leq m}  \norm{\partial^\alpha F}_{L^\infty(U)}$. We set $\norm{F}_{W^{m,\infty}(\Omega)} :=\max_{\abs{\alpha} \leq m}  \norm{\partial^\alpha F}_{L^\infty(\Omega)}$ and $\norm{f}_{W^{m,\infty}(\Omega)} := \inf_F \norm{F}_{W^{m,\infty}(\Omega)}$ where the infimum runs over all $C^m$ extensions $F$ of $f$ to some open neighbourhood of $\Omega$.
The empty map has norm zero.

\subsection{... concerning generalized intervals}

For functions $f,g: A \to \R$, $A \subseteq \R^n$, such that $f(x)<g(x)$ for all $x\in A$, abbreviated as $f<g$, define 
\begin{align*}
(f,g)_A := \, & \{(x,y)\in\R^{n+1}:x \in A, f(x) < y < g(x)\} \subseteq \R^{n+1}, \\
(-\infty,g)_A := \, & \{(x,y)\in\R^{n+1}:x \in A, y < g(x) \} \subseteq \R^{n+1}, \\
(f,\infty)_A := \, &\{(x,y)\in\R^{n+1}:x \in A, f(x) < y \} \subseteq \R^{n+1}, \\
(-\infty,\infty)_A:= \, &A \times \R \subseteq \R^{n+1}, \\
(f,f)_A := \, & \Gamma(f) \subseteq \R^{n+1}, \\
(g,f)_A := \, & (\infty, f)_A := (g,-\infty)_A := (\infty,-\infty)_A := \emptyset.
\end{align*}
Here, the symbols $\pm \infty \in \R_\infty$ are also interpreted as constant functions $\pm \infty:A\to\R_\infty$.

\subsection{... concerning neural networks}\label{sec:nnconvention}

For notions regarding feed-forward neural networks we follow \cite{petersen2025mathematicaltheorydeeplearning}. 
A \emph{ReLU neural network} $\Phi$ with \emph{architecture} $(d_0,\ldots,d_{L+1}) \in (\N^*)^{L+2}$ is a tuple $((M^{(0)},b^{(0)}),\ldots,(M^{(L)},b^{(L)}))$ of matrix-vector pairs where $L\in\N^*$, $M^{(i)} \in \R^{d_{i+1} \times d_i}, b^{(i)} \in \R^{d_{i+1}}$ for $0\leq i\leq L$. We identify $\Phi$ with its associated \emph{realization}
$$\Phi:\R^{d_0} \to \R^{d_{L+1}}, x \mapsto \Phi(x),$$ where the output is recursively calculated by 
\begin{enumerate}[label=(\roman*)]
    \item $x^{(0)}:=x$,
    \item $x^{(i)}:=\rho(M^{(i-1)}x^{(i-1)}+b^{(i-1)})$ for $i=1,\ldots,L$,
    \item $\Phi(x):=M^{(L)}x^{(L)}+b^{(L)}$.
\end{enumerate}
Here, $$\rho: \R \to \R, x \mapsto \max \lbrace 0,x \rbrace$$
denotes the \emph{ReLU activation function} and is applied coordinate-wise to each entry of the vectors. 
We use the same letter for the realization of $\Phi$ and typically identify the matrix-vector tuple $((M^{(0)},b^{(0)}),\ldots,(M^{(L)},b^{(L)}))$ as an element $\theta \in \R^{\sum_{i=0}^L d_{i+1}(d_i+1)}$, called the \emph{weights} of $\Phi$. To emphasize the dependence on $\theta$, we write $\Phi = \Phi_\theta$.
Let $\size(\Phi):= |\theta^{\neq 0}|$ be the number of non-zero weights of $\Phi$, $\width(\Phi):=\max_{i\in[L]} d_i$ and $\depth(\Phi):=L$. Further let $\norm{\Phi}_{\max}:=\norm{\theta}_{\max}$ denote the maximal absolute value of the weights of $\Phi$.

For $n\in\N^*$ and $L,W,B > 0$, let $\mathcal{NN}(n,L,W,B) \subseteq C(\R^n)$ be the class of continuous maps that are represented by ReLU neural networks $\Phi: \R^n \to [0,1]$ where 
\begin{enumerate}[label=(\roman*)]
    \item  $\norm{\Phi}_{\max} \leq B$,
    \item $\size(\Phi) \leq W$,
    \item $\depth(\Phi) \leq L$.
\end{enumerate}
If $L<1$, then $\mathcal{NN}(n,L,W,B)=\emptyset$. If $L\geq 1$ and $W<1$, then $\mathcal{NN}(n,L,W,B)=\{0\}$. Suppose that $L \geq 1$ and $W \geq 1$. Then every element of $\mathcal{NN}(n,L,W,B)$ can be represented by some ReLU neural network that has weights $\theta$ contained in the set
$$\Theta_{n,L,W,B} := \bigcup_{1\leq \ell \leq \lfloor L\rfloor}\{\theta \in [-B,B]^{\tau_\ell}:\, |\theta^{\neq 0}| \leq \lfloor W\rfloor,\,\Phi_\theta(\R^n)\subseteq [0,1] \}.$$
For each $\ell$, $\tau_\ell:=1 + \lfloor W\rfloor(n+2)+(\ell-1)(\lfloor W\rfloor^2+\lfloor W\rfloor)$ is the number of weights of the fixed architecture $(n,\lfloor W\rfloor, \ldots ,\lfloor W\rfloor,1) \in \N^{\ell+2}$, filling in dead neurons for layers with fewer than $\lfloor W\rfloor$ neurons. Note that $\Theta_{n,L,W,B}$ is a finite disjoint union of compact sets.

\section{Main results}\label{sec:mainresults}

We first state the definition of the decision sets considered in the main approximation and estimation theorems that isolate the properties of definable sets in o-minimal structures for fast approximation of their characteristic functions in $L^p$ by ReLU neural networks.

\begin{definition}
Call a set $A \subseteq \R^n$ \emph{$m$-traceable} if there exist non-empty sets $A_0,\ldots,A_n=A$ and bounded $C^m$ maps $f_i,g_i:A_{i-1} \to \R$ such that for each $i\in[n]$:
\begin{enumerate}[label=(\roman*)]
    \item $A_i \subseteq \R^i$,
    \item $f_i<g_i$ or $f_i=g_i$,
    \item $\pi_i(A_i) = A_{i-1}$ and $A_i = (f_i,g_i)_{A_{i-1}}$.
\end{enumerate}
The sequence $$(\{f_1,g_1\}, \ldots, \{f_n,g_n\})$$
is called the \emph{signature datum} of $A$ and $A_0,\ldots,A_n$ is the \emph{fundament} of $A$.

Let $A \subseteq Q_n$ be $m$-traceable.
If for each $i\in[n]$, there are $C^m$ functions $F_i, G_i:Q_{i-1} \to \R$ with $F_i|_{A_{i-1}} = f_i$ and $ G_i|_{A_{i-1}} =g_i$ so that $\norm{F_i}_{W^{m,\infty}(Q_{i-1})} \leq B$ and $\norm{G_i}_{W^{m,\infty}(Q_{i-1})} \leq B$, we say that $A$ has the \emph{$(m,B)$-extension property}.
For $B>0$, let $\mathcal{K}_{m,n}^B$ be the set of all $m$-traceable subsets of $Q_n$ with the $(m,B)$-extension property.
\end{definition}

Next, we consider finite disjoint unions of $m$-traceable sets.

\begin{definition}
Let $\ell \in \N$. Call a set $\Omega \subseteq \R^n$ \emph{$(m,\ell)$-traceable} if there are pairwise disjoint $m$-traceable sets $A^1,\ldots,A^l \subseteq \R^n$ for some $l \in \{0,\ldots,\ell\}$  such that $\Omega = A^1\cup\cdots\cup A^l$.

Let $\Omega \subseteq Q_n$ be $(m,\ell)$-traceable.
Then $\Omega$ is said to have the \emph{$(m,B,\ell)$-extension property} if there are pairwise disjoint $A^1,\ldots,A^l \in \mathcal{K}_{m,n}^B$ for some $l \in \{0,\ldots,\ell\}$ such that $\Omega = A^1\cup\cdots\cup A^l$.
Define $\mathcal{K}_{m,n}^{B,\ell}$ to be the set of all $(m,\ell)$-traceable subsets of $Q_n$ that have the $(m,B,\ell)$-extension property.
\end{definition}

The first result concerns approximation of sets in $\mathcal{K}_{m,n}^{B,\ell}$ by ReLU neural networks.

\begin{theorem}\label{thm:mainapprox}
For $n\geq 2$, $B>0$ and $\ell \in \N$, there exist constants $c=c(m,n,B,\ell,p)>0$ and $s=s(m,n,B,\ell,p)>0$ so that for every $\Omega \in \mathcal{K}_{m,n}^{B,\ell}$ and $0<\eps<1/2$, there is a ReLU neural network $\Phi_\eps^\Omega: \R^n \to [0,1]$ with 
\begin{enumerate}[label=(\roman*)]
    \item  $\norm{\Phi_\eps^\Omega}_{\max} \leq \eps^{-s}$,
    \item $\depth(\Phi_\eps^\Omega) \leq \lceil \log_2((\ell+1) n)\rceil+(3+\lceil \log_2(m)\rceil)(14+2m/n)$,
    \item $\size(\Phi_\eps^\Omega) \leq c\cdot\eps^{-p(n-1)/m}$,
\end{enumerate}
such that
$$\norm{\Phi_\eps^\Omega-\chi_\Omega}_{L^p(Q_n)} < \eps.$$
\end{theorem}

The same rate can be achieved for a certain class of piecewise smooth functions on smoothness domains given by traceable sets, see Subsection \ref{sec:piecewisesmooth}.
By Lemma \ref{lem:traceablefibres},  natural subclasses of $\mathcal{K}_{m,n}^{B,\ell}$ come from fibres of members of $\mathcal{K}_{m,k+n}^{B,\ell}$ for $k \geq 1$. Applying Theorem \ref{thm:mainapprox} in such a setting yields the following result.

\begin{corollary}\label{cor:uniformmain}
For $n \geq 2,k\geq 1, B>0, \ell \in \N$, there exist constants $c=c(m,n,B,\ell,p)>0$ and $s=s(m,n,B,\ell,p)>0$ so that for every $\Omega \in \mathcal{K}_{m,k+n}^{B,\ell}$ and $0<\eps<1/2$ and every $a \in \R^k$, there is a ReLU neural network $\Phi_\eps^a:\R^n \to [0,1]$ with 
\begin{enumerate}[label=(\roman*)]
    \item $\norm{\Phi_\eps^a}_{\max} \leq \eps^{-s}$,
    \item $\depth(\Phi_\eps^a) \leq \lceil \log_2((\ell+1) n)\rceil+(3+\lceil \log_2(m)\rceil)(14+2m/n)$,
    \item $\size(\Phi_\eps^a) \leq c\cdot\eps^{-p(n-1)/m}$,
\end{enumerate}
such that
$$\norm{\Phi_\eps^a-\chi_{\Omega(a)}}_{L^p(Q_n)} < \eps.$$
\end{corollary}

With the approximation capabilities established, we get learning rates for the class $\mathcal{K}_{m,n}^{B,\ell}$ via empirical risk minimization over the $L^p$ approximating ReLU neural networks. We measure the learning rate via the minimax error $$\eps_N(\mathcal{K}_{m,n}^{B,\ell}) := \inf_{\mathcal{A}:\Lambda_N \to L^2(Q_n)} \sup_{\Omega \in \mathcal{K}_{m,n}^{B,\ell}} \mathbb{E}_{(X_1,\ldots,X_N)\sim (\lambda^n)^N}\norm{\mathcal{A}((X_i,\chi_\Omega(X_i))_{i=1}^N)-\chi_\Omega}^2_{L^2(Q_n)},$$
where $(\lambda^n)^N$ is the product measure on $Q_n^N$ and the infimum runs over all measurable algorithms $\mathcal{A}$ that map every sample $S$ of size $N$ contained in the measure space $\Lambda_N := (Q_n \times \{0,1\})^N \subseteq (\R^{n+1})^N$ to a prediction $\mathcal{A}(S)\in L^2(Q_n)$. 

\begin{theorem}\label{thm:statlearn}
Let $n \geq 2$, $B>0,\ell \in \N$, and $\Omega \in \mathcal{K}_{m,n}^{B,\ell}$. For $N \in \N^*$, let $X_1,\ldots,X_N$ be independently sampled from the uniform distribution $\lambda^n$ on $Q_n$ and $S:= (X_i,\chi_\Omega(X_i))_{i=1}^N$. 
Further, let $\Phi_S$ be an empirical risk minimizer with respect to the hinge loss $\{0,1\}\times[0,1] \to \R, (y,\hat y) \mapsto 1-(2y-1)(2\hat y-1)$ over 
$$\mathcal{NN}(n,L,cN^{p(n-1)/(m+pn-p)},N^{sm/(m+pn-p)})$$
where $L:=\lceil \log_2((\ell+1)n)\rceil+(3+\lceil \log_2(m)\rceil)(14+2m/n)$ and $c,s>0$ only depend on $n,m,B,\ell,$ and $p$. Then
$$\mathbb{E}_S \mathbb{E}_{X\sim\lambda^n}\abs{\chi_\Omega(X) - (\chi_{[1/2,\infty)} \circ \Phi_S) (X)}^2 \lesssim N^{\kappa-m/(m+pn-p)}$$
for every $0 < \kappa \ll 1$, where the implicit constant is independent of both $N$ and $\Omega$.
\end{theorem}

Theorem \ref{thm:statlearn} yields an upper bound on the minimax error of the form
$$\eps_N(\mathcal{K}_{m,n}^{B,\ell}) \leq \sup_{\Omega \in \mathcal{K}_{m,n}^{B,\ell}} \mathbb{E}_{(X_1,\ldots,X_N)\sim 
(\lambda^n)^N}\norm{\chi_\Omega -\chi_{[1/2,\infty)}\circ\Phi_S}^2_{L^2(Q_n)} \lesssim N^{\kappa-m/(m+pn-p)}.$$

We now explain how the results above can be applied to definable collections in o-minimal structures.

\begin{definition}\label{def:omin}
An \emph{o-minimal structure} on $\R$ is a sequence $ \mathcal{S} := (\mathcal{S}_n)_{n \in \N}$ such that for all $n$:
\begin{enumerate}[label=(\roman*)]
    \item $\R^n \in\mathcal{S}_n \subseteq \mathcal{P}(\R^n)$,
    \item $A,B \in \mathcal{S}_n \implies A \cap B, A \cup B, \R^n\setminus A \in \mathcal{S}_n$,
    \item $\mathcal{S}_1$ consists of exactly the finite unions of points and open intervals in $\R$,
    \item $\{(x_1,x_2)\in\R^2: x_1<x_2\} \in \mathcal{S}_2$,
    \item $\{(x_1,x_2,x_3)\in\R^3: x_1+x_2=x_3\}, \{(x_1,x_2,x_3)\in\R^3: x_1\cdot x_2 = x_3\} \in \mathcal{S}_3$,
    \item $A \in \mathcal{S}_n \implies \R \times A, A \times \R \in \mathcal{S}_{n+1}$,
    \item  $n\geq 2 \implies \{ (x_1,\ldots,x_n) \in \R^n: x_1 =x_n \} \in \mathcal{S}_n$,
    \item $A \in \mathcal{S}_{n+1} \implies \pi_{n+1}(A) \in \mathcal{S}_n$.
\end{enumerate}
A set $A \subseteq \R^n$ is \emph{definable} (in $\mathcal{S}$) if $A \in \mathcal{S}_n$. A map $f:A \to \R$ is \emph{definable} (in $\mathcal{S}$) if its graph $\Gamma(f) \subseteq \R^{n+1}$ is definable. Call $\mathcal{H} \subseteq \mathcal{P}(\R^n)$ a \emph{definable collection} (in $\mathcal{S}$) if there is some definable set $A \subseteq \R^k \times \R^n$, $k\geq 1$, such that $\mathcal{H} = \{A(a):a\in D \}$ for $D \subseteq \R^k$ definable. 
\end{definition}

\emph{For the remainder of the paper, $\mathcal{S}$ is a fixed o-minimal structure on $\R$. Definability always refers to $\mathcal{S}$.}

For a definable collection $\mathcal{H}$ consisting of bounded definable sets, every $H\in\mathcal{H}$ is $(m,\ell)$-traceable for some uniform $\ell =\ell(m,\mathcal H)\in \N$ independent of $H$, by Lemma \ref{lem:signature}. Theorem \ref{thm:mainapprox} allows ReLU approximation of definable collections whose fibres are contained in $Q_n$ and additionally satisfy the $(m,B,\ell)$ extension property for a common $B>0$. A stronger criterion that implies the regularity condition is given in Lemma \ref{lem:traceablefibres}.

\begin{corollary}\label{cor:ominapplication}
Let $\mathcal{H} \subseteq \mathcal{K}_{m,n}^{B,\ell}$ be a definable collection for $n\geq 2$, $B>0$ and $\ell \in \N$. Then there exist constants $c=c(\mathcal{H},m,n,p)>0$, $s=s(\mathcal{H},m,n,p)>0$ and $t=t(\mathcal{H},m,n,p) > 0$ such that for every $0<\eps<1/2$ and every $H \in \mathcal{H}$, there is a ReLU neural network $\Phi_\eps^H: \R^n \to [0,1]$ with 
\begin{enumerate}[label=(\roman*)]
    \item $\norm{\Phi_\eps^H}_{\max} \leq \eps^{-s}$,
    \item $\depth(\Phi_\eps^H) \leq t$,
    \item $\size(\Phi_\eps^H) \leq c\cdot\eps^{-p(n-1)/m}$,
    \item $\norm{\Phi_\eps^H-\chi_H}_{L^p(Q_n)} < \eps$.
\end{enumerate}
In particular, the conclusion holds if $\mathcal{H}=\{A(a):a\in D\}$ for some definable $A \in \mathcal{K}_{m,k+n}^{B,\ell}$ and definable $D \subseteq Q_k$ with $k\in\N^*$.
\end{corollary}

The paper is organized as follows.
After the preliminaries in Section \ref{sec:prelim} on extensions of maps, Section \ref{sec:omin} discusses aspects of o-minimal structures including the smooth cell decomposition theorem and definable collections. The properties of definable sets are then generalized in Section \ref{sec:prep}, studying traceable sets and their boundaries. Afterwards, in Section \ref{sec:approx}, the approximating ReLU neural network class is constructed and the proof of Theorem \ref{thm:mainapprox} is given. Additionally, certain piecewise smooth functions on traceable domains are approximated by ReLU neural networks in Subsection \ref{sec:piecewisesmooth}.
Lastly, Section \ref{sec:estim} is devoted to the proof of Theorem \ref{thm:statlearn}.

\section{Preliminaries}\label{sec:prelim}

First, we give a useful lemma concerning smoothness of fibres of functions.

\begin{lemma}\label{lem:fibremaps}
Let $0 \leq i\leq n$, $A \subseteq \R^n$, and $f:A \to \R$ be a $C^m$ map.
Then for each $a \in \R^i$, the map $f(a):A(a) \to \R$ is $C^m$. 
Further, if $A$ and all fibres $A(a)$ are measurable and for any $U \subseteq \R^n$ open
\begin{equation}\label{eq:compcond}
A\cap U \neq \emptyset \quad \implies \quad \lambda^n(A \cap U) > 0,
\end{equation} 
then $\norm{f(a)}_{W^{m,\infty}(A(a))} \leq \norm{f}_{W^{m,\infty}(A)}$ for $A(a)\neq \emptyset$.
\end{lemma}
\begin{proof}
As $f$ is $C^m$, there is some open $U \supseteq A$ and a $C^m$ map $F: U \to \R$ such that $F|_A = f$.
Let $a \in \R^i$. Then $F(a)|_{A(a)} = f(a)$ and we claim that $U(a) \supseteq A(a)$ is open in $\R^{n-i}$ and that $F(a): U(a) \to \R$ witnesses that $f(a)$ is $C^m$.
To this end, we distinguish two cases. 
In case $a \in \R^i \setminus \pi_n^{\leq i}(U)$, $U(a) = \emptyset$ is open and $F(a)=f(a):\emptyset \to \R$ is the empty map.
If $a \in  \pi_n^{\leq i}(U)$, then for each $b=(b_{i+1},\ldots,b_n) \in \R^{n-i}$ such that $(a,b) \in U$, there is some $\delta=\delta(b) > 0$ with
$$(a,b) \in \bigtimes_{j=1}^i(a_j-\delta,a_j+\delta) \times \bigtimes_{j=i+1}^n (b_j-\delta,b_j+\delta) \subseteq U.$$
Hence $b \in \bigtimes_{j=i+1}^n (b_j-\delta,b_j+\delta) \subseteq U(a)$, so $U(a)$ is open. 

In all cases, $F(a): U(a) \to \R$ satisfies $\partial^{\tilde{\alpha}} (F(a)) = (\partial^\alpha F)(a)$ for $\alpha \in \{0\}^ i\times\N^{n-i}$, $\abs{\alpha} \leq m$, $\tilde{\alpha}:=(\alpha_{i+1}, \ldots, \alpha_n)\in\N^{n-i}$, so $F(a)$ is $C^m$. 
Further, for $A(a)\neq \emptyset$,
$$\norm{\partial^{\tilde{\alpha}} (F(a))}_{L^\infty(A(a))}\leq \sup_{y\in A(a)}\{\abs{(\partial^\alpha F)(a)(y)}\} \leq \sup_{(x,y)\in A}\{\abs{\partial^\alpha F(x,y)}\}= \norm{\partial^\alpha F}_{L^\infty(A)} \leq \norm{F}_{W^{m,\infty}(A)},$$
where assumption \eqref{eq:compcond} is used to replace the essential supremum by a supremum.
Hence $\norm{F(a)}_{W^{m,\infty}(A(a))} \leq \norm{F}_{W^{m,\infty}(A)}$ and the second part follows by taking the infimum over all such extensions $F$.
\end{proof}

We recall results concerning extended real valued functions from variational analysis \cite{rockafellar1998variational}, Chapter 1.

\begin{definition}
Let $f:\R^n \to \R_\infty$. Then $f$ is \emph{lower semi-continuous} (lsc) at $x \in \R^n$ if $f(x) = \liminf_{z \to x}f(z)$ and \emph{upper semi-continuous} (usc) at $x \in \R^n$ if $f(x) = \limsup_{z \to x}f(z)$. Here, $$\liminf_{z \to x}f(z) := \sup_{\delta >0} \inf_{z\in B_\delta(x)}f(z), \quad \limsup_{z \to x}f(z) := \inf_{\delta >0} \sup_{z\in B_\delta(x)}f(z).$$
We say that $f$ is \emph{lsc} (resp. \emph{usc}) if $f$ is lsc (resp. usc) at every $x \in \R^n$.

We let $\Gamma^\uparrow(f) := \{(x,y) \in \R^n \times \R:f(x)\leq y\}$ be the \emph{epigraph} of $f$ and $\Gamma^\downarrow(f) := \{(x,y) \in \R^n \times \R:f(x)\geq y\}$  be the \emph{hypograph} of $f$.
\end{definition}

\begin{remark}\label{rem:lim}
Let $f:\R^n \to \R_\infty$. If $\liminf_{z \to x} f(z) = \limsup_{z \to x} f(z)$ for every $x \in \R^n$, then $f$ is continuous.
Also, $f$ is lsc if and only if $-f$ is usc as $\liminf_{z \to x} -f(z) = -\limsup_{z \to x} f(z)$. We have 
$\Gamma^\downarrow(-f) = \{(x,-y): (x,y)\in\Gamma^\uparrow(f)\}$ and $\Gamma^\uparrow(-f) = \{(x,-y):(x,y) \in \Gamma^\downarrow(f)\}$.
\end{remark}

\begin{proposition}\label{prop:closed}
Let $f:\R^n \to \R_\infty$.
\begin{enumerate}[label=(\roman*)]
    \item If $f$ is lsc, then $\Gamma^\uparrow(f)$ is closed.
    \item If $f$ is usc, then $\Gamma^\downarrow(f)$ is closed.
\end{enumerate}
\end{proposition}

\begin{proof}
(i) Follows from \cite{rockafellar1998variational}, Theorem 1.6. 

(ii) Follows from (i) and Remark \ref{rem:lim} replacing $f$ by $-f$.
\end{proof}

\begin{lemma}\label{lem:context}
Let $A \subseteq \R^n$. If $f:A \to \R$ is continuous, then $\underline{f}:\R^n \to \R_\infty$, where
\begin{equation*}
\underline{f}(x) := 
\begin{cases*}
\sup_{\delta>0}\inf_{z \in B_\delta(x)\cap A} f(z), \quad x\in\cl(A) \\
\infty, \quad x\in\R^n\setminus\cl(A)
\end{cases*}
\end{equation*}
 is the unique lsc extension of $f$ such that every lsc map $g:\R^n \to \R_\infty$ with $g|_A \leq f$ satisfies $g \leq \underline{f}$.
Also, the map $\overline{f}: \R^n \to \R_\infty$, where
\begin{equation*}
\overline{f}(x) := 
\begin{cases*}
\inf_{\delta>0} \sup_{z\in B_\delta(x)\cap A} f(z), \quad x\in\cl(A) \\
- \infty, \quad x\in\R^n\setminus\cl(A)
\end{cases*}
\end{equation*}
is the unique usc extension of $f$ such that every usc map $h:\R^n \to \R_\infty$ with $h|_A \geq f$ satisfies $h \geq \overline{f}$.
Moreover, $\Gamma^\uparrow(\underline{f}) \subseteq \cl(A) \times \R$ and $\Gamma^\downarrow(\overline{f}) \subseteq \cl(A) \times \R$.
\end{lemma}

\begin{proof}
For each $x\in A$, the continuity of $f$ yields
$$\sup_{\delta>0}\inf_{z \in B_\delta(x)\cap A} f(z) = \inf_{\delta>0} \sup_{z\in B_\delta(x)\cap A} f(z) = f(x),$$
so both $\underline{f}$ and $\overline{f}$ extend $f$.

To show lower semi-continuity of $\underline{f}$, let $y \in \R$. Suppose that $x \in \cl(A)$. If $y < \underline{f}(x)$, then $y<\inf_{z \in B_{\delta'}(x)\cap A} f(z)$ for some $\delta'>0$. For every $z' \in B_{\delta'/2}(x)\cap \cl(A)$, we have $B_{\delta'/2}(z') \subseteq B_{\delta'}(x)$ and
$$\underline{f}(z') = \sup_{\delta>0}\inf_{z \in B_{\delta}(z')\cap A} f(z) \geq \inf_{z \in B_{\delta'/2}(z')\cap A} f(z)\geq \inf_{z \in B_{\delta'}(x)\cap A} f(z) > y.$$
Hence $\{x\in\R^n:\underline{f}(x) > y\} = \{x\in\cl(A):\underline{f}(x) > y\} \cup (\R^n\setminus\cl(A))$ is open for every $y\in\R$. 
This yields that if $\underline{f}(x_0)>y_0$ for $x_0\in\R^n$ and $y_0\in\R$, then $$\liminf_{z \to x_0} \underline{f}(z) = \sup_{\delta>0}\inf_{z \in B_\delta(x_0)}\underline{f}(z) \geq \inf_{z \in B_{\delta'}(x_0)}\underline{f}(z) \geq y_0$$
for some $\delta'>0$. So for $\underline{f}(x_0)>-\infty$ and $y_0 \to \underline{f}(x_0)$, we get $\underline{f}(x_0)=\liminf_{z \to x_0} \underline{f}(z)$. In case $\underline{f}(x_0)=-\infty$, $\underline{f}$ is automatically lsc. 
Similarly, $\overline{f}(x)=\limsup_{z \to x} \overline{f}(z)$ by replacing $f$ by $-f$ in the preceding discussion.

Let $g$ and $h$ be as claimed. Then $\underline{f} \geq g$ and $\overline{f} \leq h$ on $\R^n \setminus \cl(A)$. For every $x \in \cl(A)$ and $\delta>0$, $B_\delta(x)\cap A \neq \emptyset$ and so $$g(x) = \liminf_{z \to x} g(z) \leq \sup_{\delta>0}\inf_{z \in B_\delta(x)\cap A} f(z) = \underline{f}(x)$$ and $$h(x) = \limsup_{z\to x} h(z) \geq \inf_{\delta>0} \sup_{z\in B_\delta(x)\cap A} f(z) = \overline{f}(x).$$
Uniqueness uses the properties just shown and the last claim follows from the definition of $\underline{f}$ and $\overline{f}$.
\end{proof}

\begin{definition}
For $A\subseteq \R^n$ and a continuous function $f:A \to \R$, we call $\underline{f}$ the \emph{lsc envelope} of $f$ and $\overline{f}$ the \emph{usc envelope} of $f$, where $\underline{f}$ and $\overline{f}$ are defined as in Lemma \ref{lem:context}.
\end{definition}

\section{O-minimal structures}\label{sec:omin}

In Euclidean space, o-minimal structures naturally extend the notion of semi-algebraic sets while preserving many desirable properties including the ability to decompose definable sets into finitely many \emph{cells}. These structural results exclude infinite one-variable oscillations and space filling curves. In fact, not even the integers are definable in any o-minimal structure because they have infinitely many connected components. 
Nonetheless, rather interesting and flexible o-minimal structures exist, as described in the examples below.

Functions and sets contained in o-minimal structures not only have many regularity properties on bounded regions, but also behave controllably on unbounded domains. For the interested reader, we mention that the germs at $+\infty$ of one-variable definable maps in an o-minimal structure give rise to a Hardy field \cite{miller2012basics,aschenbrenner2017asymptotic}.

This section heavily relies on \cite{van1998tame}, a source for most fundamental results of o-minimality. We record consequences of Definition \ref{def:omin} of an o-minimal structure on the ordered field of real numbers. The name refers to the constraint that definable subsets of $\R$ are finite unions of intervals and points. Since the ordering is definable, intervals are necessarily definable, hence these structures are order-minimal. The classical definition as in \cite{van1998tame} does not require addition and multiplication to be definable, leading to the consideration of the structure only including the ordering and the structure of semi-linear sets. However, the field structure allows us to define an intrinsic notion of smoothness which is essential for this paper.

We briefly treat an alternative way of viewing o-minimal structures through the lens of model theory.
Originally, sets contained in an o-minimal structure are defined by first order formulas in the underlying language of a certain model theoretic structure interpreting a total ordering. We sketch the formalism of first order logic and refer to \cite{van1998tame}, Chapter 1, for further details.

Take sets of symbols $X=\{x_n\}_{n}$ (variables), $\mathcal{F}\supseteq \R \cup \{+,\cdot\}$ (function symbols) and $\mathcal{R} \supseteq \{=,<\}$ (relation symbols). Assign to each $f \in \mathcal{F}$ an \emph{arity} $n_f\in \N$ which is thought of as the number of inputs of $f$. Similarly, every $R \in \mathcal{R}$ has an arity $n_R\in\N$. For example, $n_==n_<=n_+=n_\cdot=2$ and $n_r=0$ for all $r\in\R$ (strictly speaking, we distinguish between 'names' of real numbers as function symbols and their interpretation as a real number).

The symbols in $\mathcal{F} \cup \mathcal{R}$ are interpreted as functions $f:\R^{n_f} \to \R$, and relations $R \subseteq \R^{n_R}$ respectively where $+,\cdot,<$ are interpreted as addition, multiplication, and ordering. \emph{Terms} are built inductively by successively substituting simpler terms in function symbols: each variable $x\in X$ and every constant $r \in \R$ is a term and if $f \in \mathcal{F}$ and $t_1,\ldots,t_{n_f}$ are terms, then $f(t_1,\ldots,t_{n_f})$ is a term.

A \emph{basic formula} is a statement about certain terms being in some relation. More precisely, if $R\in \mathcal{R}$ and $t_1,\ldots, t_{n_R}$ are terms, then $R(t_1,\ldots,t_{n_R})$ is a basic formula. The \emph{formulas} consist of basic formulas and are closed under boolean combinations and quantification over variables: for formulas $\phi$ and $\psi$, also $\phi \lor \psi$, $\phi \land \psi,\neg \phi,\forall x \phi$ and $\exists x \phi$ are formulas for every $x\in X$ with their usual interpretation. If a formula has occurrences of variables not bound by quantifiers, real numbers can be substituted into the variables and the truth value of the instance of the formula can be recorded. If $\phi(\bar{x})$ is a formula with free variables $\bar{x} \in X^n$ for suitable $n$, then 
$$A:= \{a \in \R^n: \phi(a) \text{ holds}\}$$
is the set \emph{defined by $\phi$}. Here, $\phi(a)$ is a sentence that has a binary truth value. A set that is defined by a formula is called a \emph{definable} set.

If the subsets of $\R$ defined by formulas in $\mathcal{F} \cup \mathcal{R}$ are finite unions of points and intervals, then the definable sets in $\mathcal{F} \cup \mathcal{R}$ form an o-minimal structure on $\R$. 
The connection between the axioms from Definition \ref{def:omin} and the logical operations combines syntax and semantics as follows.
\begin{enumerate}[label=(\roman*)]
    \item Unions, intersections and complements of sets correspond to boolean connectives $\lor,\land$ and $\neg$ applied to their defining formulas.
    \item Projections onto certain coordinates correspond to existential quantification over the remaining coordinate variables.
    \item Diagonals correspond to equalities among variables. The graphs of addition and multiplication as well as the order relation are given by basic formulas.
\end{enumerate}

On the other hand, an o-minimal structure (in the sense of Definition \ref{def:omin}) yields a model theoretic structure when viewing every definable set (in the sense of Definition \ref{def:omin}) as the interpretation of a relation symbol in the language having for each definable subset $A \subseteq \R^n$ a distinguished relation symbol $R_A$ with arity $n$ which gets interpreted as $A$. This follows from the closure properties of o-minimal structures (in the sense of Definition \ref{def:omin}).

We continue with some examples of well known o-minimal structures.

\begin{example}
The \emph{semi-algebraic} sets are described by the definable sets using the symbols $\mathcal{F} \cup \mathcal{R} =\R \cup \{+,\cdot,=,<\}$, an o-minimal structure where definable subsets of $\R^n$ consist of finite unions of sets of the form $\{x \in \R^n:p(x)=0, p_1(x) > 0,\ldots, p_\ell(x) > 0\}$ for polynomials $p,p_1,\ldots,p_\ell \in \R[X_1,\ldots,X_n]$, $\ell\in\N$ \cite{van1998tame}. A function is semi-algebraic if its graph is semi-algebraic. Every ReLU neural network realization is semi-algebraic \cite{kratsios2026feedforwardneuralnetworkdefinable}.
\end{example}

Other interesting examples include $\R_{\exp}$, where $\mathcal{F} \cup \mathcal{R}= \R \cup \{\exp,+,\cdot,=,<\}$, interpreting $\exp:\R\to\R$ as the global exponential function, and $\R_{\an}$, containing for each $n$ all \emph{restricted analytic functions} $\R^n \to \R$ which are given as the restrictions to $[-1,1]^n$ of real analytic functions $U\to\R$, $U\supseteq [-1,1]^n$ open, extended by zero outside of $[-1,1]^n$. The common expansion $\R_{\an,\exp}$ of $\R_{\exp}$ and $\R_{\an}$ is again an o-minimal structure \cite{van1994elementary}.
A more general source of o-minimal structures is provided by restricted functions from quasi-analytic Denjoy-Carleman classes satisfying the properties considered in \cite{rolin2003quasianalytic}. Although classical Weierstrass preparation might fail for these function classes, resolution of singularities and normalization methods control the number of connected components of definable sets.

Also, expanding the ordered field of real numbers by \emph{Pfaffian functions} yields an o-minimal structure \cite{wilkie1999theorem}.
A Pfaffian function $f:\R^n \to \R$ is constructed via a finite sequence $f_1,\ldots,f_k$ of $C^1$ functions where $f=f_k$ such that $$\partial^j f_i \in \R[x_1,\ldots,x_n,f_1,\ldots,f_i]$$ for all $i\in[k],j\in[n]$. For example, $x \mapsto \arctan(x)$ is Pfaffian witnessed by the sequence $f_1(x) = 1/(x^2+1), f_2(x) =\arctan(x)$, as verified by computing $f_1'(x) = -2x/(x^2+1)^2 = -2xf_1(x)^2$ and $f_2'(x) = 1/(x^2+1)=f_1(x)$. Many practically used activation functions of neural networks are Pfaffian.

\subsection{Cells in o-minimal structures}

In this section, we study definable sets and state the cell decomposition theorem that generalizes the shape of definable subsets of $\R$ prescribed by the axioms: for $n\geq 2$, definable sets in $\R^n$ are built from $C^m$-cells that behave like intervals and points in $\R$. To state the inductive definition of a $C^m$-cell, we specialize the notion of smoothness.

\begin{definition}
Let $A \subseteq \R^n$ be definable and $f:A \to \R$ definable. Then $f$ is said to be \emph{definably $C^m$} if there is some definable open set $U \supseteq A$ and a definable map $F:U \to \R$ such that $F|_A = f$ and $F$ is $C^m$, i.e. for each $\alpha \in \N^n$, $\abs{\alpha} \leq m$, the map $\partial^\alpha F: U \to \R$ exists and is continuous. Per convention, we also say that the maps $A \to \R_\infty$, $x\mapsto \infty$ and $A \to \R_\infty$, $x\mapsto -\infty$ are definably $C^m$ maps.
\end{definition}

\begin{remark}
Every definably $C^m$ map $f:A \to \R$ for $A \subseteq \R^n$ is $C^m$, witnessed by the same (definable) open $U$ and (definable) map $F$ extending $f$.
\end{remark}

\begin{definition}\label{def:cell}
Let $(i_1,\ldots,i_n) \in \{0,1\}^n$. We define inductively when $C \subseteq \R^n$ is an \emph{$(i_1,\ldots,i_n)$-$C^m$-cell}.
\begin{enumerate}[label=(\roman*)]
    \item The only $()$-$C^m$-cell is $\R^0:=\{*\}$.
    \item Suppose that $C \subseteq \R^n$ is an $(i_1, \ldots, i_n)$-$C^m$-cell and $f,g:C \to \R_\infty$ are definably $C^m$ maps satisfying $f=g$ or $f<g$. Then $(f,g)_C \subseteq \R^{n+1}$ is an
    \begin{enumerate}
        \item  $(i_1, \ldots, i_n,0)$-$C^m$-cell if $f=g:C \to \R$.
        \item  $(i_1, \ldots, i_n,1)$-$C^m$-cell if $f<g$.
    \end{enumerate}
\end{enumerate}
The \emph{$C^m$-cells} of $\mathcal{S}$ are all $(i_1,\ldots,i_n)$-$C^m$-cells for every $n$ and $(i_1,\ldots,i_n) \in \{0,1\}^n$. 
The \emph{signature} of a $C^m$-cell $C \subseteq \R^n$ is the unique sequence $(i_1,\ldots,i_n) \in \{0,1\}^n$ so that $C$ is an $(i_1,\ldots,i_n)$-$C^m$-cell.
Further, the \emph{dimension} $\dim(C) \in \{0,\ldots,n\}$ of $C$ is defined as $\dim(C) := i_1+\cdots+i_n$. 
\end{definition}

\emph{Let $A \subseteq \R^n$ be a definable set}. We consider partitions of $A$ into $C^m$-cells that are defined inductively.

\begin{definition}
A \emph{$C^m$-cell decomposition} $\mathcal{C}\subseteq \mathcal{P}(A)$ of $A$ is a partition of $A$ into finitely many $C^m$-cells such that 
$$\mathcal{C}_i:=\{\pi_n^{\leq i}(C): C \in \mathcal{C}\}$$
is a $C^m$-cell decomposition of $\pi_n^{\leq i}(A)$ for all $i \in[n-1]$, where a $C^m$-cell decomposition of $\pi_n^{\leq 1}(A)$ is a partition of $\pi_n^{\leq 1}(A)$ into finitely many open intervals and points.
We say that $\mathcal{C}$ is \emph{compatible} with a definable set $A' \subseteq A$ if $A'$ is a union of $C^m$-cells in $\mathcal{C}$.
\end{definition}

We state the central $C^m$-cell decomposition theorem, the main structural result about definable maps and sets \cite{van1998tame}, Chapter 7, (3.3), Exercise 2, or Theorem (3.2) for the $C^1$ version. Another source is \cite{coste1999introduction}, Theorem 6.6 and Theorem 6.7.

\begin{theorem}\label{thm:omincell}
For $\ell\in\N$ and definable sets $A_1,\ldots, A_\ell \subseteq A$, and any definable map $f:A \to \R$, there is a $C^m$-cell decomposition $\mathcal{C}$ of $A$ compatible with $A_i$ for each $i\in[\ell]$ and such that $f|_C$ is $C^m$ for every $C^m$-cell $C \in \mathcal{C}$.
\end{theorem}

\begin{remark}
For an infinite collection $\mathcal{A}$ of definable sets, the number of $C^m$-cells needed to decompose members of $\mathcal{A}$ might be arbitrarily large, but there are certain natural classes of definable sets, called definable collections, which admit uniform upper bounds on the number of $C^m$-cells needed to decompose each individual member, see Subsection \ref{sec:deffam}. 
\end{remark}

\begin{lemma}\label{lem:defextfun}
Let $C \subseteq \R^n$ be a $C^m$-cell and $f:C \to \R$ definably $C^m$ and $0\leq i \leq n$.
Then for each $a \in \pi_n^{\leq i}(C)$, the map $f(a):C(a) \to \R$ is definably $C^m$.
\end{lemma}
\begin{proof}
As $f$ is definably $C^m$, there is some definable open $U \supseteq C$ and a definable map $F: U \to \R$ such that $F|_C = f$ and for all $\alpha \in \N^n$, $\abs{\alpha} \leq m$, the map $\partial^\alpha F: U \to \R$ exists and is continuous.
Let $a \in \R^i$. Then $U(a) \supseteq C(a)$ is definable and as in the classical case in Lemma \ref{lem:fibremaps}, $U(a)$ is open in $\R^{n-i}$ and $F(a): U(a) \to \R$ witnesses that $f(a)$ is $C^m$. As $F(a)$ and $U(a)$ are definable, this gives that $f(a)$ is definably $C^m$.
\end{proof}

\subsection{Definable collections}\label{sec:deffam}

Natural potentially infinite classes of definable sets are given by definable collections that consist of fibres of definable sets. The advantage of definable collections is that the number of $C^m$-cells needed to decompose an individual fibre is uniformly bounded. This also reflects the fact that o-minimal structures have NIP. We make the definition of a definable collection (Definition \ref{def:omin}) more explicit. 

\begin{definition}
Let both $A \subseteq \R^n$ and $D \subseteq \R^i$ be definable where $i \in [n-1]$. Call
$$\mathcal{A}_{i,D} :=\{A(a) \subseteq \R^{n-i}:a\in D\} \subseteq \mathcal{P}(\R^{n-i})$$
the \emph{definable collection coming from $A$ with parameter space $D \subseteq \R^i$}.
\end{definition}

\begin{remark}
Suppose that $i\in[n-1]$ and $\mathcal A_{i,D}\subseteq \mathcal{P}(\R^{n-i})$ is the definable collection coming from some definable $A \subseteq \R^n$ with parameter space $D \subseteq \R^i$. The coordinate-wise map $h(x)=x/\sqrt{4+4x^2}$ is an injective definable map $\R^i \to Q_i$. Let $A'$ be the image of $A$ under the definable map 
$$\R^i \times \R^{n-i} \to \R^i \times \R^{n-i}, \quad (a,x)\mapsto(h(a),x).$$
Then $\mathcal A'_{i,h(D)}=\mathcal A_{i,D}$ where $\mathcal A'_{i,h(D)}$ is the definable collection coming from $A' \subseteq \R^n$ with parameter space $h(D) \subseteq Q_i$. If $D\neq\emptyset$, we can fix $a_0\in D$ and let $A'':=(A'\cap (h(D)\times \R^{n-i}))\cup ((Q_i\setminus h(D)) \times A(a_0))$ so that $A''(a)=A(a_0)$ for every $a\in Q_i\setminus h(D)$ to arrange $\mathcal A''_{i,Q_i}=\mathcal A'_{i,h(D)}=\mathcal A_{i,D}$.
\end{remark}

\emph{Let $A \subseteq \R^n$ be definable and $A_k:=\pi_n^{\leq k}(A)$ for every $0\leq k \leq n$}. The next statement illustrates that $C^m$-cell decompositions of high dimensional definable sets yield $C^m$-cell decompositions of fibres thereof, and thus give a uniform bound on the number of $C^m$-cells needed to decompose each fibre \cite{van1998tame}.

\begin{lemma}\label{lem:uniformdeffam}
Let $\mathcal{C}$ be a $C^m$-cell decomposition of $A$ and $i\in[n-1]$. Then for each $a \in \R^i$,  $$\mathcal{C}(a) := \{C(a):C \in \mathcal{C},C(a)\neq \emptyset\}$$is a $C^m$-cell decomposition of $A(a)$ with $|\mathcal{C}(a)| \leq |\mathcal{C}|$.
\end{lemma}
\begin{proof}
By Lemma  \ref{lem:defextfun} and the inductive definition of a $C^m$-cell, it follows that if $C \subseteq \R^n$ is a $C^m$-cell and $a \in \R^i$, then $C(a) \subseteq \R^{n-i}$ is a $C^m$-cell or empty. For $C,C'\in\mathcal{C}$, we get 
$$C\cap C'=\emptyset \quad \implies \quad C(a)\cap C'(a) = (C\cap C')(a) = \emptyset.$$
Since $A= \bigcup_{C \in \mathcal{C}} C$, it follows that $A(a) = (\bigcup_{C\in\mathcal{C}} C)(a) = \bigcup_{C\in\mathcal{C}} C(a)$, so $\mathcal{C}(a)$ partitions $A(a)$. 
Projecting yields $\pi_{n-i}^{\leq j}(C(a)) = (\pi_n^{\leq i+j}(C))(a)$ for $1\leq j<n-i$, hence $(\mathcal{C}(a))_j = \mathcal{C}_{i+j}(a)$ is a partition of $A_{i+j}(a)$, which follows by replacing $A$ by $A_{i+j}$ above.
This shows that $\mathcal{C}(a)$ is a cell decomposition of $A(a)$. The size bound immediately follows.
\end{proof}

\begin{example}
Consider the semi-algebraic set $A:=\{(a,b,c,d,e,f,x,y)\in\R^{6+2}:ax^2+bxy+cy^2+dx+ey+f=0\}$ and the definable collection $\mathcal{A}_{6,\R^6}$ coming from $A$. The fibres $A(a,b,c,d,e,f)$ for $(a,b,c,d,e,f)\in\R^6$ describe all conic sections and degenerate cases thereof.
\end{example}

\begin{example}
The graphs of realizations of ReLU neural networks with a fixed architecture $(d_0,\ldots,d_{L+1}) \in (\N^*)^{L+2}$, where $d_0 = n$ and $d_{L+1}=1$, as defined in Section \ref{sec:nnconvention} can be viewed as a definable collection $\mathcal{A}_{j,D}$ with $D := \R^j$ and $A \subseteq D \times \R^n \times \R$, where $j:=\sum_{i=0}^L d_{i+1}(d_i+1)$ and
$$(\theta,x,y) \in A \quad :\iff \quad \Phi_\theta(x) = y.$$
Moreover, the hypothesis class $\{x\in Q_n:{\Phi_\theta(x)\geq 1/2}\}_{\theta\in D}$ of classifiers for a binary classification task can be viewed as a definable collection $\mathcal{A}'_{j,D}$ for $A' \subseteq D \times Q_n$ with $$(\theta,x) \in A' \quad :\iff \quad \Phi_\theta(x)\geq 1/2.$$
More general classes of neural networks with fixed architecture and definable activation functions can be described by definable collections \cite{kratsios2026feedforwardneuralnetworkdefinable}. Since definable collections have finite VC dimension, no fixed architecture using a definable activation function can uniformly approximate the class of continuous functions on any non-degenerate compact box.
\end{example}

\section{Traceable sets}\label{sec:prep}

In this section, we study $C^m$-cells and finite disjoint unions thereof from a more general perspective using boundedness and smoothness of the defining functions rather than their definability. This separates the application within o-minimal structures from the more general approximation theory of traceable sets, making the required properties for our approach explicit. We start with motivating examples.

\begin{example}\label{ex:trace1}
The unique traceable set in $\R^0$ is $\{*\}$.
Every singleton set $\{a\}$ for $a \in \R^n$ is $m$-traceable with signature datum satisfying $$f_i(x_1,\ldots,x_{i-1})=g_i(x_1,\ldots,x_{i-1})=a_i$$ and fundament $A_i = \{(a_1,\ldots,a_i)\} \subseteq \R^i$.

Any open box $\bigtimes_{i=1}^n(a_i,b_i)$ for $-\infty <a_i<b_i<\infty$ is $m$-traceable with signature datum satisfying $f_i(x_1,\ldots,x_{i-1})=a_i$ and $g_i(x_1,\ldots,x_{i-1})=b_i$. Its fundament is given by $A_i = \bigtimes_{j=1}^i(a_j,b_j) \subseteq \R^i$. 
So for any $C^m$ function $F:(-1/2,1/2)^n \to (-1/2,1/2)$, the sets $(-1/2,F)_{(-1/2,1/2)^n}$, $(F,F)_{(-1/2,1/2)^n}$, and $(F,1/2)_{(-1/2,1/2)^n}$ are $m$-traceable. Here, the constants are interpreted as constant functions on the domain $(-1/2,1/2)^n$.
\end{example}

\begin{example}\label{ex:trace2}
Let 
$$A:= \{(x_1,x_2,x_3) \in \R^3:0<x_1<1,-1<x_2<1,0<x_3<\arctan(x_2/x_1)/\pi+1/2\} \subseteq \R^3.$$
Then $A$ is $m$-traceable with fundament $A_0 =\{*\},A_1=(0,1),A_2=(0,1)\times (-1,1), A_3 = A$ and signature datum satisfying
\begin{enumerate}[label=(\roman*)]
    \item $f_1 = 0,\quad g_1=1$,
    \item $f_2(x_1)=-1,\quad g_2(x_1)=1$,
    \item $f_3(x_1,x_2) = 0,\quad g_3(x_1,x_2)=\arctan(x_2/x_1)/\pi+1/2$.
\end{enumerate}
The function $g_3$ is depicted in Figure \ref{fig:trace2}.
\end{example}

\begin{example}\label{ex:trace3}
The set 
$$A := \{(x,y,z)\in\R^3:0 < x < 1, y=\sin(1/x), z=\cos(1/x)\} \subseteq \R^3$$
depicted in Figure \ref{fig:trace3} is $m$-traceable with fundament $A_0 =\{*\},A_1=(0,1),A_2=\{(x_1,x_2)\in\R^2:0<x_1<1,x_2=\sin(1/x_1)\},A_3=A$ and
signature datum satisfying 
\begin{enumerate}[label=(\roman*)]
    \item $f_1 = 0,\quad g_1=1$,
    \item $f_2(x_1) = g_2(x_1) = \sin(1/x_1)$,
    \item $f_3(x_1,x_2) = g_3(x_1,x_2)=\cos(1/x_1)$.
\end{enumerate}
\end{example}

\emph{Let $A$ be $m$-traceable with fundament $A_0,\ldots,A_n$ and signature datum $(\{f_i,g_i\}_{i\in[n]})$}. Fix $i\in[n]$. Then for $m'\in [m]$, $A_i$ is $m'$-traceable with signature datum $(\{f_1,g_1\}, \ldots, \{f_i,g_i\})$ and fundament $A_0,\ldots, A_i$. Furthermore,
$$(x_1,\ldots,x_i) \in A_i \iff \bigwedge_{j=1}^i x_j \in (\tilde{f}_j(x_1,\ldots,x_{j-1}),\tilde{g}_j(x_1,\ldots,x_{j-1}))_{\R^0},$$
where $\tilde{f}_j,\tilde{g}_j:\R^{j-1} \to \R_\infty$ extend $f_j,g_j$ by $\tilde{f}_j(x_1,\ldots,x_{j-1}):=\infty$ and $\tilde{g}_j(x_1,\ldots,x_{j-1}) := -\infty$ for $(x_1,\ldots,x_{j-1}) \in \R^{j-1}\setminus A_{j-1}$.

For $0\leq j<i\leq n$, we have $A_j = \pi_i^{\leq j}(A_i)$ by the representation above. Furthermore, for every $a \in A_j$,
$$(x_{j+1},\ldots,x_i) \in A_i(a) \iff \bigwedge_{k=j+1}^i x_k \in (\tilde{f}_k(a,x_{j+1},\ldots,x_{k-1}),\tilde{g}_k(a,x_{j+1},\ldots,x_{k-1}))_{\R^0}.$$
So $A_i(a)$ is also $m$-traceable with signature datum $(\{f_{j+1}(a), g_{j+1}(a)\},\ldots, \{f_i(a),g_i(a)\})$ and fundament $A_j(a), \ldots, A_i(a)$. The fundament and the signature datum of $A$ are unique. 

The next lemma provides a bound on the Lebesgue measure of $A_i$.

\begin{lemma}\label{lem:measuretraceable}
For each $i\in[n]$, $A_i$ is measurable and
$$\lambda^i(A_i) \leq  \prod_{j=1}^i \norm{g_j-f_j}_{L^\infty(A_{j-1})}.$$
In particular, if $f_j = g_j$ for some $j\in[n]$, then $\lambda^i(A_i) = 0$ for $i\geq j$.
\end{lemma}

\begin{proof}
For $n=0$, the statement is vacuous, so we assume $n\geq 1$. We have that $\lambda^0(A_0)=1$ and that $A_1$ is measurable with $\lambda^1(A_1) = \lambda^1((f_1,g
_1)_{A_0}) = g_1-f_1 = \norm{g_1-f_1}_{L^\infty(\R^0)}$.
Suppose that $i\in[n], i \geq 2$, and the claim holds for $A_1,\ldots,A_{i-1}$. 
The measurability of $A_i$ follows as in the proof of \cite{krapp2025measurabilityfundamentaltheoremstatistical}, Lemma A.9.
Since $A_i = (f_i,g_i)_{A_{i-1}}$, Fubini's theorem gives
$$\lambda^i(A_i)= \int_{A_{i-1}} g_i(x)-f_i(x) dx \leq \lambda^{i-1}(A_{i-1})\norm{g_i-f_i}_{L^\infty(A_{i-1})} \leq \prod_{j=1}^i \norm{g_j-f_j}_{L^\infty(A_{j-1})}.$$
The second claim follows.
\end{proof}

\begin{corollary}\label{cor:measuretraceable}
Let $0\leq j < i \leq n$. For $a \in A_j$, we have that
$$\lambda^{i-j}(A_i(a)) \leq  \prod_{k=j+1}^i \norm{g_k(a)-f_k(a)}_{L^\infty(A_{k-1}(a))}.$$
Furthermore, if $f_k(a) = g_k(a)$ for some $j<k\leq i$, then $\lambda^{i-j}(A_i(a)) = 0$.
\end{corollary}
\begin{proof}
Apply Lemma \ref{lem:measuretraceable} to the $m$-traceable set $A_i(a)$, whose signature datum is $$(\{f_{j+1}(a), g_{j+1}(a)\},\ldots, \{f_i(a),g_i(a)\}).$$
\end{proof}

\begin{definition}
Define the \emph{extended signature datum} of $A$ by $(\{\underline{f}_i,\overline{g}_i\}_{i\in[n]})$ where $\underline{f}_i$ is the lsc envelope of $f_i$ and $\overline{g}_i$ is the usc envelope of $g_i$ for each $i\in[n]$.
Call $A$ \emph{extendable} if for every $i\in[n]$, $\underline{f}_i|_{\cl(A_{i-1})}$ and $\overline{g}_i|_{\cl(A_{i-1})}$ are continuous. We say that $A$ is \emph{$m$-extendable} if for every $i\in[n]$, there are continuous functions $f_i^\alpha,g_i^\alpha:\cl(A_{i-1})\to\R$ for $\alpha \in \N^{i-1}$, $\abs{\alpha} \leq m$, such that $f_i^0=\underline{f}_i|_{\cl(A_{i-1})}$, $g_i^0=\overline{g}_i|_{\cl(A_{i-1})}$, and, for every $\alpha$,
$$\sup_{\substack{x,y \in \cl(A_{i-1})\\0<\abs{x-y}\leq \delta}}\abs{x-y}^{\abs{\alpha}-m} \abs{f_i^\alpha(x)-\sum_{\substack{\beta \in \N^{i-1}\\\abs{\beta} \leq m-\abs{\alpha}}} f_i^{\alpha+\beta}(y)(x-y)^\beta/\beta!}\to 0$$
and 
$$\sup_{\substack{x,y \in \cl(A_{i-1})\\0<\abs{x-y}\leq \delta}}\abs{x-y}^{\abs{\alpha}-m} \abs{g_i^\alpha(x)-\sum_{\substack{\beta \in \N^{i-1}\\ \abs{\beta} \leq m-\abs{\alpha}}} g_i^{\alpha+\beta}(y)(x-y)^\beta/\beta!}\to 0$$
as $0<\delta \to 0$.
\end{definition}

\begin{remark}\label{rem:Whitney}
Note that if $\underline{f}_i|_{\cl(A_{i-1})}$ and $\overline{g}_i|_{\cl(A_{i-1})}$ are $C^m$ for each $i\in[n]$, then $A$ is $m$-extendable, witnessed by the derivatives up to order $m$ of the $C^m$ extensions of $\underline{f}_i|_{\cl(A_{i-1})}$ and $\overline{g}_i|_{\cl(A_{i-1})}$ to an open neighbourhood of $\cl(A_{i-1})$ restricted to $\cl(A_{i-1})$ as a consequence of Taylor's theorem. By the Whitney extension theorem \cite{whitney1934analytic}, Theorem I, the reverse is also true: if $A$ is $m$-extendable, then all $\underline{f}_i|_{\cl(A_{i-1})}$ and $\overline{g}_i|_{\cl(A_{i-1})}$ admit global $C^m$ extensions that are real analytic on $\R^{i-1}\setminus \cl(A_{i-1})$. Also definable versions of the Whitney extension theorem are known \cite{aschenbrenner2019whitney}.
We only work with the $C^m$ extensions of signature data for $m$-extendable traceable sets, keeping in mind that extendability is an intrinsic property of the closure of a traceable set.
\end{remark}

With the extended signature datum of $A$, an explicit description of its topological closure $\cl(A)$ is possible.

\begin{lemma}\label{lem:closuredescription}
For each $i \in [n]$, we have 
$$\cl(A_i) \subseteq \{(x_1,\ldots,x_i) \in \R^i: \underline{f}_1 \leq x_1 \leq \overline{g}_1, \ldots, \underline{f}_i(x_1,\ldots,x_{i-1}) \leq x_i \leq \overline{g}_i(x_1,\ldots,x_{i-1})\}.$$
Moreover, if $A$ is extendable, then equality holds.
\end{lemma}

\begin{proof}
The case $n=0$ is vacuous, so assume $n\geq 1$. Considering $A_1=(f_1,g_1)_{A_0}$, we see that $\cl(A_1)=\{x_1\in\R:f_1\leq x_1\leq g_1\}=\{x_1\in\R:\underline{f}_1\leq x_1\leq \overline{g}_1\}$.

Suppose that $i<n$ and the statement holds for $A_1,\ldots,A_i$.
It is enough to show that $\cl(A_{i+1}) \subseteq \Gamma^\uparrow(\underline{f}_{i+1}) \cap \Gamma^\downarrow(\overline{g}_{i+1})$ since
\begin{align*}
\underline{f}_{i+1}^{-1}(\R) = \overline{g}_{i+1}^{-1}(\R)=\cl(A_i) \subseteq \{ & (x_1,\ldots,x_i) \in \R^i: \quad  \underline{f}_1 \leq x_1 \leq \overline{g}_1, \ldots, \\ & \underline{f}_i(x_1,\ldots,x_{i-1}) \leq x_i \leq \overline{g}_i(x_1,\ldots,x_{i-1})\}
\end{align*}
and $\Gamma^\uparrow(\underline{f}_{i+1}),\Gamma^\downarrow(\overline{g}_{i+1}) \subseteq \cl(A_i) \times \R$ by Lemma \ref{lem:context}.
Since $A_{i+1} \subseteq \Gamma^\uparrow(\underline{f}_{i+1})\cap\Gamma^\downarrow(\overline{g}_{i+1})$ and $\Gamma^\uparrow(\underline{f}_{i+1})\cap\Gamma^\downarrow(\overline{g}_{i+1})$ is closed by Proposition \ref{prop:closed}, we get that $\cl(A_{i+1}) \subseteq \Gamma^\uparrow(\underline{f}_{i+1})\cap\Gamma^\downarrow(\overline{g}_{i+1})$.

For the second part, assume that $A$ is extendable and that equality holds for $A_1,\ldots,A_i$. Suppose that $(x_1,\ldots,x_{i+1}) \in \R^{i+1}$ satisfies
$$\underline{f}_1 \leq x_1 \leq \overline{g}_1, \ldots, \underline{f}_{i+1}(x_1,\ldots,x_i) \leq x_{i+1} \leq \overline{g}_{i+1}(x_1,\ldots,x_i).$$
Then $(x_1,\ldots,x_i) \in \cl(A_i)$ by assumption and there is some $\theta \in [0,1]$ such that 
$$x_{i+1} = \theta \underline{f}_{i+1}(x_1,\ldots,x_i) + (1-\theta)\overline{g}_{i+1}(x_1,\ldots,x_i).$$
Take a sequence $\theta_j \in (0,1)$ with $\theta_j \to \theta$ and a sequence $(x^j_1,\ldots,x^j_i)_{j\in\N} \subseteq A_i$ converging to $(x_1,\ldots,x_i)$. Set
$$x^j_{i+1} := \theta_j \underline{f}_{i+1}(x^j_1,\ldots,x^j_i) + (1-\theta_j)\overline{g}_{i+1}(x^j_1,\ldots,x^j_i).$$ Then $(x^j_1,\ldots,x^j_{i+1})\in A_{i+1}$ and by continuity of $\underline{f}_{i+1}|_{\cl(A_i)}$ and $\overline{g}_{i+1}|_{\cl(A_i)}$, these points converge to $(x_1,\ldots,x_{i+1})$, which therefore lies in $\cl(A_{i+1})$.

\end{proof}

Without extendability, equality in Lemma~\ref{lem:closuredescription} may hold or fail, as the following examples show.

\begin{example}
Let $A$ be as in Example \ref{ex:trace2}. Then $A$ is not extendable since at $x_1=x_2=0$, there is no continuous extension of $g_3$ to the closure of its domain, as depicted in Figure \ref{fig:trace2}. However, 
$$\cl(A)=\{(x_1,x_2,x_3):0\leq x_1 \leq 1,-1 \leq x_2 \leq 1, 0 \leq x_3 \leq \overline{g}_3(x_1,x_2)\}$$
where $\overline{g}_3(0,x_2)=0$ for $x_2 \in [-1,0)$ and $\overline{g}_3(0,x_2)=1$ for $x_2 \in [0,1]$.
\end{example}

\begin{figure}[htbp]
\centering
\begin{tikzpicture}
\begin{axis}[
    width=0.86\linewidth,
    height=8cm,
    view={50}{25},
    axis lines=box,
    xlabel={$x_1$}, ylabel={$x_2$}, zlabel={$x_3$},
    label style={font=\small},
    tick label style={font=\small},
    xticklabel style={xshift=-5pt, yshift=-2pt},
    yticklabel style={xshift=5pt, yshift=-2pt},
    ylabel style={rotate=-90},
    zlabel style={rotate=-90},
    xmin=0, xmax=1, ymin=-1, ymax=1, zmin=0, zmax=1,
    xtick={0,0.5,1}, ytick={-1,0,1}, ztick={0,0.5,1},
    grid=major,
    grid style={gray!20},
    colormap={boundaryblue}{color(0cm)=(blue!15); color(1cm)=(blue!65)},
    legend style={
        at={(0.5,-0.19)}, anchor=north,
        draw=none, font=\small,
        legend columns=2, /tikz/every even column/.append style={column sep=1em}
    },
]
\addplot3[
    surf, shader=interp,
    domain=0:1, y domain=-90:90,
    samples=31, samples y=81,
    forget plot,
]
({x*cos(y)/max(cos(y),abs(sin(y)))},
 {x*sin(y)/max(cos(y),abs(sin(y)))},
 {y/180+0.5});
\addlegendimage{area legend, draw=blue!65, fill=blue!25}
\addlegendentry{Graph closure}
\addplot3[red!75!black, very thick, no marks]
    coordinates {(0,0,0) (0,0,1)};
\addlegendentry{Limiting fibre}
\end{axis}
\end{tikzpicture}
\caption{The closure of the graph of
$g_3(x_1,x_2)=\arctan(x_2/x_1)/\pi+1/2$ on
$(0,1)\times(-1,1)$. The highlighted segment is its fibre
$\{(0,0,x_3):0\leq x_3\leq1\}$ above $(0,0)$.
Thus $g_3$ has no continuous extension to that point.}
\label{fig:trace2}
\end{figure}

\begin{example}
The set $A$ from Example \ref{ex:trace3} is not extendable. 
The lsc extension of $f_2$ and the usc extension of $g_2$ are given by $\underline{f}_2(0)= -1$ and $\overline{g}_2(0) = 1$.
The lsc extension of $f_3$ and the usc extension of $g_3$ are given by $\underline{f}_3(0,x_2) = -\sqrt{1-x_2^2}$ and $\overline{g}_3(0,x_2) = \sqrt{1-x_2^2}$ for $x_2 \in [-1,1]$. Since $\cl(A) = A \cup \{(0,x_2,x_3):x_2^2+x_3^2=1\} \cup \{(1,\sin(1),\cos(1))\}$ does not include the whole filled disk $\{(0,x_2,x_3):x_2^2+x_3^2\leq 1\}$, equality in Lemma \ref{lem:closuredescription} fails, see Figure \ref{fig:trace3}.
\end{example}

\begin{figure}[htbp]
\centering
\begin{tikzpicture}
\begin{axis}[
    width=0.86\linewidth,
    height=8cm,
    view={50}{25},
    axis lines=box,
    xlabel={$x_1$}, ylabel={$x_2$}, zlabel={$x_3$},
    label style={font=\small},
    tick label style={font=\small},
    xticklabel style={xshift=-5pt, yshift=-2pt},
    yticklabel style={xshift=5pt, yshift=-2pt},
    ylabel style={rotate=-90},
    zlabel style={rotate=-90},
    xmin=0, xmax=1, ymin=-1, ymax=1, zmin=-1, zmax=1,
    xtick={0,0.5,1}, ytick={-1,0,1}, ztick={-1,0,1},
    grid=major,
    grid style={gray!20},
    legend style={
        at={(0.5,-0.19)}, anchor=north,
        draw=none, font=\small,
        legend columns=2, /tikz/every even column/.append style={column sep=1em}
    },
]
\addplot3[
    blue!70!black, thin, no marks,
    variable=\t, domain=1:200, samples=3001, samples y=0,
]
({1/\t}, {sin(deg(\t))}, {cos(deg(\t))});
\addlegendentry{Finite truncation}
\addplot3[
    red!75!black, thick, dashed, no marks,
    variable=\t, domain=0:360, samples=181, samples y=0,
]
({0}, {sin(\t)}, {cos(\t)});
\addlegendentry{Limiting circle}
\end{axis}
\end{tikzpicture}
\caption{The blue curve shows the finite truncation
$\{(1/t,\sin t,\cos t):1\leq t\leq200\}$ of the closure of $A$;
its endpoint at $t=1$ belongs to $\operatorname{cl}(A)\setminus A$.
The dashed circle $\{(0,x_2,x_3):x_2^2+x_3^2=1\}$ is the entire
limiting fibre as $x_1\to0^+$, rather than the filled disk.}
\label{fig:trace3}
\end{figure}

Using the description of $\cl(A)$ from Lemma \ref{lem:closuredescription} allows us to explicitly analyse the frontier of $A$.

\begin{corollary}\label{cor:frontier}
The frontier of $A$ is contained in a union of $2n$ (not necessarily distinct) sets
\begin{align*}
\fr(A) \subseteq & \{\underline{f}_1 = x_1 \leq \overline{g}_1, \ldots, \underline{f}_n(x_1,\ldots,x_{n-1})\leq x_n \leq \overline{g}_n(x_1,\ldots,x_{n-1})\} \\
\cup & \{\underline{f}_1 \leq x_1 = \overline{g}_1, \ldots, \underline{f}_n(x_1,\ldots,x_{n-1})\leq x_n \leq \overline{g}_n(x_1,\ldots,x_{n-1})\} \\
\cup & \cdots \\
\cup & \{\underline{f}_1 \leq x_1 \leq \overline{g}_1, \ldots, \underline{f}_n(x_1,\ldots,x_{n-1}) = x_n \leq \overline{g}_n(x_1,\ldots,x_{n-1})\} \\
\cup & \{\underline{f}_1 \leq x_1 \leq \overline{g}_1, \ldots, \underline{f}_n(x_1,\ldots,x_{n-1})\leq x_n = \overline{g}_n(x_1,\ldots,x_{n-1})\}.
\end{align*}
Also, $\lambda^n(\fr(A))= 0$.
\end{corollary}

\begin{proof}
If $n=0$, then $A=\R^0$ and its frontier is empty so we can assume $n\geq1$. For $x \in \fr(A)$, let $i \in [n]$ be minimal such that 
$$f_i(x_1,\ldots,x_{i-1}) < x_i < g_i(x_1,\ldots,x_{i-1})$$
does not hold. 
By minimality of $i$, $(x_1,\ldots,x_{i-1}) \in \dom(f_i) = \dom(g_i)$.
As Lemma \ref{lem:closuredescription} implies that $\underline{f}_i(x_1,\ldots,x_{i-1}) \leq x_i \leq \overline{g}_i(x_1,\ldots,x_{i-1})$ and $\pi_n^{<i}(x) \in A_{i-1}$, it holds that $\underline{f}_i(x_1,\ldots,x_{i-1}) = x_i$ or $x_i = \overline{g}_i(x_1,\ldots,x_{i-1})$.

For the second part, note that $\lambda^0(\fr(A_0))=\lambda^0(\emptyset)=0$ and $\fr(A_1) \subseteq \{f_1,g_1\}$, so $\lambda(\fr(A_1))=0$.
Suppose that $\lambda^{i-1}(\fr(A_{i-1})) = 0$ for $i\in[n]$. Then 
$$\fr(A_i) = \fr((f_i,g_i)_{A_{i-1}}) \subseteq (\fr(A_{i-1}) \times I) \cup \Gamma(f_i) \cup \Gamma(g_i),$$
where $I:=[a,b]:=[\min_{x\in\cl(A_{i-1})} \underline{f}_i(x),\max_{x\in\cl(A_{i-1})} \overline{g}_i(x)]$ and both $\Gamma(f_i)$ and $\Gamma(g_i)$ are null-sets. Moreover, Fubini's Theorem gives
$$\lambda^i(\fr(A_{i-1}) \times I)= \int_{\fr(A_{i-1})} \abs{b-a} dx = \lambda^{i-1}(\fr(A_{i-1}))(b-a)=0,$$
hence $\lambda^i(\fr(A_i))=0$.
\end{proof}

\emph{Assume that $A \subseteq Q_n$ is $m$-extendable}. We associate to $A$ the constant 
$$B_A^m := \inf\max_{i\in[n]} \{\norm{F_i}_{W^{m,\infty}(Q_{i-1})}, \norm{G_i}_{W^{m,\infty}(Q_{i-1})}\}$$
where $F_i$ and $G_i$ range over all $C^m$ extensions of $f_i$ and $g_i$ to $Q_{i-1}$ respectively. Note that by Remark \ref{rem:Whitney}, $B_A^m$ is well defined and finite.

\emph{Let $\Omega \subseteq \R^n$ be $(m,\ell)$-traceable}. Then $\Omega = A^1 \cup \cdots \cup A^l$ for some $l \leq \ell$ and pairwise disjoint $m$-traceable $A^j$ for $j\in[l]$. Call $\Omega$ \emph{$(m,\ell)$-extendable} if $l$ and $A^j$ can be chosen so that $A^j$ is $m$-extendable for every $j\in[l]$. 
\emph{Assume that $\Omega \subseteq Q_n$ is $(m,\ell)$-extendable}.
Associate to $\Omega$ the constant 
$$B_\Omega^{m,\ell} := \inf\{\max_{j\in[l]} B_{A^j}^m: \quad \Omega = A^1\cup\cdots\cup A^l,l\leq\ell, A^j \in \bigcup_{B>0}\mathcal{K}_{m,n}^B \text{ pairwise disjoint}\},$$
where the maximum over an empty set is zero.
Next, we study the classes $\mathcal{K}_{m,n}^B$ and $\mathcal{K}_{m,n}^{B,\ell}$ for $B>0$ and $\ell \in \N$. 
Note that $\mathcal{K}_{m,n}^{B,0} = \{\emptyset\}$ for all $B>0$. For $n=1$, we have 
$$\bigcup_{\ell\in\N}\mathcal{K}_{m,1}^{B,\ell} =  \{A \cap [-B,B] \cap Q_1:\, A \in \mathcal{S}_1\}.$$
More generally, every semi-linear subset of $Q_n$ belongs to $\bigcup_{B>0,\ell\in\N}\mathcal{K}_{m,n}^{B,\ell}$ by \cite{van1998tame}, Chapter 1, Corollary (7.6). In particular, for $n\geq 1$ graphs of and bands between ReLU neural networks $\R^{n-1} \to \R$ intersected with $Q_n$ belong to $\bigcup_{B>0,\ell\in\N}\mathcal{K}_{m,n}^{B,\ell}$.

\begin{lemma}\label{lem:extprop}
The following statements hold.
\begin{enumerate}[label=(\roman*)]
    \item $A \in \mathcal{K}_{m,n}^B$ for every $B > B_A^m$.
    \item $\Omega \in \mathcal{K}_{m,n}^{B,\ell}$ for every $B > B_\Omega^{m,\ell}$.
    \item $\mathcal{K}_{m,n}^B \subseteq \mathcal{K}_{m',n}^{B'} \subseteq \mathcal{K}_{m',n}^{B',\ell} \subseteq \mathcal{K}_{m'',n}^{B'',\ell'}$ for $0<B\leq B'\leq B'',1\leq\ell \leq \ell', m \geq m'\geq m''\geq 1$.
\end{enumerate}
\end{lemma}

\begin{proof}
(i) Since $A$ is $m$-extendable, it follows for each $i \in [n]$ that $\underline{f}_i|_{\cl(A_{i-1})}$ and $\overline{g}_i|_{\cl(A_{i-1})}$ are $C^m$ and by the Whitney extension theorem \cite{whitney1934analytic}, there are global $C^m$ extensions $F_i:\R^{i-1} \to \R$ of $\underline{f}_i|_{\cl(A_{i-1})}$ and $G_i:\R^{i-1}\to\R$ of $\overline{g}_i|_{\cl(A_{i-1})}$. Restricting $F_i$ and $G_i$ to $Q_{i-1}$ yields a bound on their Sobolev norms since the continuous functions $\abs{\partial^\alpha F_i},\abs{\partial^\alpha G_i}:Q_{i-1} \to \R$ attain their maximal values on their compact domains. By definition of $B_A^m$, it is possible to choose $F_i$ and $G_i$ in a way such that 
$$B_A^m \leq \max_{i\in[n]} \{ \norm{F_i}_{W^{m,\infty}(Q_{i-1})}, \norm{G_i}_{W^{m,\infty}(Q_{i-1})}\} < B.$$ So $A$ has the $(m,B)$-extension property.

(ii) By (i) and the definition of $B_\Omega^{m,\ell}$, we get that for any $B > B_\Omega^{m,\ell}$, there are $A^1,\ldots,A^l \in \mathcal{K}_{m,n}^B$ such that $\Omega = A^1\cup\cdots\cup A^l$, so $\Omega \in \mathcal{K}_{m,n}^{B,\ell}$.

(iii) Note that $\mathcal{K}_{m,n}^B \subseteq \mathcal{K}_{m,n}^{B,1}$ by definition. The other inclusions follow from the respective definition.
\end{proof}

 The next lemma concerns fibres of elements of $\mathcal{K}_{m,k+n}^{B,\ell}$ and will be used for definable collections.

\begin{lemma}\label{lem:traceablefibres}
Let $k\in\N^*$ and $\Omega \in \mathcal{K}_{m,k+n}^{B,\ell}$. Then $\Omega(a) \in \mathcal{K}_{m,n}^{B,\ell}$ for each $a\in Q_k$.
\end{lemma}

\begin{proof}
Fix $a \in Q_k$. By our assumption, there are pairwise disjoint $A^1,\ldots,A^{l'} \in \mathcal{K}_{m,k+n}^B$ for some ${l'}\leq \ell$ such that
$\Omega = A^1\cup\cdots\cup A^{l'}$. Let $J_a = \{j \in [l']: A^j(a) \neq \emptyset\}$. Replacing $\Omega$ by $\tilde{\Omega} :=\bigcup_{j\in J_a}A^j\in\mathcal{K}_{m,k+n}^{B,\ell}$ and re-indexing yields $\tilde{\Omega}=A^1\cup\cdots\cup A^l$ for some $l\leq l'$ and $\Omega(a) = \tilde{\Omega}(a) = A^1(a) \cup \cdots \cup A^l(a)$ where $A^j(a) \neq \emptyset$ for all $j \in [l]$. 

For $i \in [l]$, let $(\{f^i_j,g^i_j\}_{j\in[k+n]})$ be the signature datum of $A^i$.
For each $j\in[k+n]$, there are $C^m$ maps $F_j^i,G_j^i:Q_{j-1}\to\R$ extending $f_j^i$ and $g_j^i$ respectively such that 
$$\max_{j\in[k+n]}\max\{\norm{F_j^i}_{W^{m,\infty}(Q_{j-1})},\norm{G_j^i}_{W^{m,\infty}(Q_{j-1})}\}\leq B.$$
Then, for $j>k$, the maps $F_j^i(a)$ and $G_j^i(a)$ extending $f_j^i(a)$ and $g_j^i(a)$ satisfy 
$$\max_{j>k}\max\{\norm{F_j^i(a)}_{W^{m,\infty}(Q_{j-1-k})},\norm{G_j^i(a)}_{W^{m,\infty}(Q_{j-1-k})}\}\leq B,$$
invoking Lemma \ref{lem:fibremaps}. Since $(\{f_j^i(a),g_j^i(a)\}_{j>k})$ is the signature datum of $A^i(a)$, we get $A^i(a) \in \mathcal{K}_{m,n}^B$.
Hence $A^i(a) \in \mathcal{K}_{m,n}^B$ for every $i\in[l]$, so $\Omega(a) \in \mathcal{K}_{m,n}^{B,\ell}$.
\end{proof}

Derivatives of bounded smooth functions may be unbounded near the boundary of their domains, so an $(m,\ell)$-traceable subset of $Q_n$ need not be $(m,\ell)$-extendable. Nevertheless, we obtain inner approximation with the same piece bound and a finite extension bound that may depend on the accuracy.

\begin{lemma}\label{lem:nottoofar}
Let $\ell \in \N$ and $\Omega \subseteq Q_n$ be an $(m,\ell)$-traceable set. For every $\eps > 0$, there is $\Omega^\eps\in \mathcal{K}_{m,n}^{B,\ell}$ for some $B=B(\eps)>0$ with $\Omega^\eps \subseteq \Omega$ and 
$$\lambda^n(\Omega \setminus \Omega^\eps) \leq \eps.$$
\end{lemma}

\begin{proof}
If $n=0$, take $\Omega^\eps=\Omega$ and $B=1$. Henceforth we assume $n\geq1$.
For $\ell=0$, the result is trivial, so we assume that $\ell>0$. Let $\eps > 0$ and pairwise disjoint $m$-traceable sets $A^1,\ldots,A^l$ for $l\leq \ell$ with fundaments $A_0^j,\ldots,A_n^j$ and signature data $(\{f_i^j,g_i^j\}_{i\in[n]})$ for $j\in[l]$ be given such that $\Omega=A^1\cup\cdots\cup A^l$. If $l=0$, we trivially have $\Omega = \emptyset$, so we can also assume $l>0$.
We show that there is some $B>0$ and $m$-traceable sets $(A^j)'\subseteq A^j$ for each $j\in[l]$ such that $(A^j)'\in\mathcal{K}_{m,n}^B$ and $\Omega^\eps := (A^1)'\cup\cdots\cup (A^l)'$ satisfies the conclusion of the Lemma. 

Fix $j\in[l]$. If there is some $i\in[n]$ such that $f_i^j = g_i^j$, then Lemma \ref{lem:measuretraceable} yields that $\lambda^n(A^j) = 0$ and we can take $(A^j)' := \{x\}$ for any $x \in A^j$. Then the signature datum of $(A^j)'$ has constant coordinate extensions for $\{x\}$ as in Example \ref{ex:trace1}, and $\{x\}\in\mathcal{K}_{m,n}^{B_j}$ for $B_j:=1/2$ witnessed by the constant extensions of the coordinates of $x$.

For notational convenience, we set $f_i:=f_i^j$ and $g_i:=g_i^j$ for $i\in[n]$, and $A:=A^j$.
Suppose that $f_i<g_i$ for all $i\in[n]$. For $0 < \delta < 1/2$ and $i\in[n]$, define $f^\delta_i, g^\delta_i: A_{i-1} \to \R$ by $$f^\delta_i(x) :=(1-\delta)f_i(x)+\delta g_i(x)$$ and 
$$g^\delta_i(x) := \delta f_i(x) + (1-\delta)g_i(x).$$
Then $f_i < f^\delta_i < g^\delta_i < g_i$ and recursively setting $A_i^\delta := (f^\delta_i|_{A_{i-1}^\delta},g^\delta_i|_{A_{i-1}^\delta})_{A_{i-1}^\delta}$ for $i\in[n]$ with $A_0^\delta = A_0$ yields the $m$-traceable set $A^\delta := A_n^\delta$ with signature datum $(\{f_1^\delta|_{\R^0},g_1^\delta|_{\R^0}\},\ldots,\{f_n^\delta|_{A_{n-1}^\delta},g_n^\delta|_{A_{n-1}^\delta}\})$. 

We have $\cl(A_0^\delta) = \{*\} \subseteq A_0$ and $\cl(A_1^\delta) = [f_1^\delta,g_1^\delta] \subseteq A_1$. Suppose that $\cl(A_{i-1}^\delta) \subseteq A_{i-1}$ for $i\in[n]$. Then 
$$\cl(A_i^\delta)\subseteq\{(x,t):x\in\cl(A_{i-1}^\delta),\ f_i^\delta(x)\leq t\leq g_i^\delta(x)\}\subseteq A_i.$$
It follows that for each $i\in[n]$, the maps $\underline{f}_i^\delta|_{\cl(A_{i-1}^\delta)}$ and $\overline{g}_i^\delta|_{\cl(A_{i-1}^\delta)}$ are $C^m$, witnessed by their $C^m$ extensions $f_i^\delta$ and $g_i^\delta$. Thus $A^\delta$ is $m$-extendable by Remark \ref{rem:Whitney}. Furthermore, $A^\delta$ has the $(m,B_j)$-extension property for 
$$B_j:=\max_{i\in[n]} \{\norm{F_i^\delta}_{W^{m,\infty}(Q_{i-1})}, \norm{G_i^\delta}_{W^{m,\infty}(Q_{i-1})} \},$$
for some global $C^m$ extensions $F_i^\delta$ and $G_i^\delta$  of $\underline{f}_i^\delta|_{\cl(A_{i-1}^\delta)}$ and $\overline{g}_i^\delta|_{\cl(A_{i-1}^\delta)}$, using the Whitney extension theorem. Also, $$\norm{f^\delta_i-f_i}_{L^\infty(A_{i-1})} = \norm{g_i-g^\delta_i}_{L^\infty(A_{i-1})} = \delta \norm{g_i-f_i}_{L^\infty(A_{i-1})}.$$
Applying Lemma \ref{lem:measuretraceable} to $A_i^\delta$ yields 
$$\lambda^i(A_i^\delta) \leq  \prod_{j=1}^i \norm{g_j^\delta-f_j^\delta}_{L^\infty(A^\delta_{j-1})} \leq \prod_{j=1}^i  \norm{g_j-f_j}_{L^\infty(A^\delta_{j-1})} \leq \prod_{j=1}^i  \norm{g_j-f_j}_{L^\infty(A_{j-1})}.$$
For $A_1 = (f_1,g_1)_{A_0}$, we explicitly have
$$\lambda^1(A_1\setminus A_1^\delta) = \lambda^1((f_1,f_1^\delta] \cup [g_1^\delta,g_1))= 2\delta (g_1 - f_1)=:\beta_1.$$
Next, if $(x_1,\ldots,x_{i+1}) \in A_{i+1} \setminus A_{i+1}^\delta$ for $0 \leq i < n$, then the following cases can occur:
\begin{enumerate}[label=(\roman*)]
    \item $(x_1,\ldots,x_i) \in A_i\setminus A_i^\delta$ and $f_{i+1}(x_1,\ldots,x_i) < x_{i+1} < g_{i+1}(x_1,\ldots,x_i)$, or
    \item $(x_1,\ldots,x_i) \in A_i^\delta$ and $f_{i+1}(x_1,\ldots,x_i) <x_{i+1} \leq f^\delta_{i+1}(x_1,\ldots,x_i)$, or
    \item  $(x_1,\ldots,x_i) \in A_i^\delta$ and $g_{i+1}^\delta(x_1,\ldots,x_i) \leq x_{i+1} < g_{i+1}(x_1,\ldots,x_i)$.
\end{enumerate}
So
\begin{align*}
\lambda^{i+1}(A_{i+1} \setminus A_{i+1}^\delta) \leq & \lambda^i(A_i\setminus A_i^\delta)\norm{g_{i+1}-f_{i+1}}_{L^\infty(A_i\setminus A_i^\delta)} + \lambda^i(A_i^\delta) \norm{g_{i+1}-g^\delta_{i+1}}_{L^\infty(A_i^\delta)} \\
& + \lambda^i(A_i^\delta) \norm{f^\delta_{i+1}-f_{i+1}}_{L^\infty(A_i^\delta)} \\
\leq & \lambda^i(A_i\setminus A_i^\delta)\norm{g_{i+1}-f_{i+1}}_{L^\infty(A_i)} + 2\lambda^i(A_i^\delta) \delta\norm{g_{i+1}-f_{i+1}}_{L^\infty(A_i)} \\
\leq & \lambda^i(A_i\setminus A_i^\delta)\norm{g_{i+1}-f_{i+1}}_{L^\infty(A_i)}+2\delta\prod_{j=1}^{i+1}\norm{g_j-f_j}_{L^\infty(A_{j-1})} \\
& =: \lambda^i(A_i\setminus A_i^\delta) \alpha_{i+1} + \beta_{i+1}.
\end{align*}
This yields that
\begin{align*}
\lambda^n(A\setminus A^\delta) & \leq \sum_{j=1}^n \beta_j \prod_{i=j+1}^n \alpha_i\\
& \leq \sum_{j=1}^n 2\delta \prod_{k=1}^j \norm{g_k-f_k}_{L^\infty(A_{k-1})}\prod_{k=j+1}^n \norm{g_k-f_k}_{L^\infty(A_{k-1})}\\
& \leq 2n\delta \prod_{k=1}^n\norm{g_k-f_k}_{L^\infty(A_{k-1})} \leq \eps/\ell 
\end{align*}
by recursively plugging in bounds for $\lambda^i(A_i\setminus A_i^\delta)$ and taking $\delta \leq \eps/(2\ell n\prod_{j=1}^n\norm{g_j-f_j}_{L^\infty(A_{j-1})})$.
Setting $(A^j)' := A^\delta$ and repeating the argument above for each $j\in[l]$ yields
$$\lambda^n(\Omega\setminus\Omega^\eps) = \sum_{j=1}^l\lambda^n(A^j\setminus (A^j)') \leq \ell \max_j \lambda^n(A^j\setminus (A^j)') \leq \ell \eps/\ell = \eps,$$
where $\Omega^\eps \in \mathcal{K}_{m,n}^{B,\ell}$ for $B:=\max_{j\in[l]} B_j$.
\end{proof}
\begin{remark}
In the proof of Lemma \ref{lem:nottoofar}, we have that for each $j\in[n]$, $\norm{g_j-f_j}_{L^\infty(A_{j-1})} \leq 1$ because $-1/2\leq f_j(x) \leq g_j(x) \leq 1/2$ for $x\in A_{j-1}$, so these factors could be omitted. However, the proof actually shows that general $(m,\ell)$-traceable sets are close to $(m,\ell)$-extendable traceable sets, where the factor $\prod_{j=1}^n\norm{g_j-f_j}_{L^\infty(A_{j-1})}$ needs to be accounted for.
\end{remark}

\subsection{Connection between traceable and definable sets}

The next lemma connects traceable sets with definable sets.

\begin{lemma}\label{lem:signature}
Every bounded $(i_1,\ldots,i_n)$-$C^m$-cell $C \subseteq \R^n$ is $m$-traceable with fundament $C_i=\pi^{\leq i}_n(C)$. Moreover, for every bounded definable set $A \subseteq \R^n$, there is some $\ell \geq 0$ such that $A(a)$ is $(m,\ell)$-traceable for all $a \in \R^i$ and $0 \leq i \leq n$.
\end{lemma}

\begin{proof}
By induction on $n$. For $n=0$ we have $A=\{*\}$ and for $n=1$, a disjoint union of $\ell \in \N$ intervals and points is $(m,\ell)$-traceable.
Let $C$ be a bounded $(i_1,\ldots,i_n)$-$C^m$-cell and assume that bounded, definably $C^m$ functions $f,g:C\to\R$ with either $f=g$ or $f<g$ are given. Then $f$ and $g$ are in particular $C^m$. Inductively, $C$ is $m$-traceable with signature datum $(\{f_i,g_i\}_{i\in[n]})$ and fundament $C_0,\ldots,C_n$. It follows that the cell $(f,g)_C$ has signature datum  $(\{f_i,g_i\}_{i\in[n]},\{f,g\})$ and fundament $C_0,\ldots,C_n,(f,g)_C$.
By Theorem \ref{thm:omincell}, every definable set is a disjoint union of finitely many $C^m$-cells and
Lemma \ref{lem:uniformdeffam} yields that the number of $C^m$-cells needed to decompose fibres of definable sets is uniformly bounded.
\end{proof}

The following corollary collects the properties of $C^m$-cells that were derived for $m$-traceable sets in Section \ref{sec:prep}.

\begin{corollary}\label{cor:indcell}
Let $n\geq 1$ and $C \subseteq \R^n$ be a bounded $C^m$-cell viewed as an $m$-traceable set with fundament $C_0,\ldots,C_n$ and signature datum $(\{f_1,g_1\}, \ldots, \{f_n,g_n\})$. Then for every $0 \leq i<n$,
\begin{enumerate}[label=(\roman*)]
    \item $C_i$ is a $C^m$-cell with fundament $C_0,\ldots,C_i$ and signature datum $(\{f_1,g_1\}, \ldots, \{f_i,g_i\})$,
    \item $C_j = \pi_i^{\leq j}(C_i)$ and $C_j(a)=\pi_{i-k}^{\leq j-k}(C_i(a))$ for all $0 \leq k<j < i$ and $a \in \R^k$,
    \item For $a \in C_i$, the fibre $C(a)$ above $a$ is a $C^m$-cell with signature datum $$(\{f_{i+1}(a),g_{i+1}(a)\}, \ldots, \{f_n(a),g_n(a)\}),$$
\end{enumerate}
Suppose that $C$ is extendable. Then
\begin{enumerate}[label=(\roman*), start=4]
    \item The closure of $C$ can be described by
    \begin{align*}
        \cl(C) =  \{ & (x_1,\ldots,x_n)\in\R^n: \underline{f}_1 \leq x_1 \leq \overline{g}_1, \ldots, \\ &\underline{f}_n(x_1,\ldots,x_{n-1}) \leq x_n  \leq \overline{g}_n(x_1,\ldots,x_{n-1})\}.
    \end{align*}
    \item The frontier of $C$ satisfies
    \begin{align*}
        \fr(C)  \subseteq & \{\underline{f}_1 = x_1 \leq \overline{g}_1, \ldots, \underline{f}_n(x_1,\ldots,x_{n-1})\leq x_n \leq \overline{g}_n(x_1,\ldots,x_{n-1})\} \\
        &  \cup \{\underline{f}_1 \leq x_1 = \overline{g}_1, \ldots, \underline{f}_n(x_1,\ldots,x_{n-1})\leq x_n \leq \overline{g}_n(x_1,\ldots,x_{n-1})\} \\
        & \cup \cdots \\
        & \cup \{\underline{f}_1 \leq x_1 \leq \overline{g}_1, \ldots, \underline{f}_n(x_1,\ldots,x_{n-1}) = x_n \leq \overline{g}_n(x_1,\ldots,x_{n-1})\} \\
        & \cup \{\underline{f}_1 \leq x_1 \leq \overline{g}_1, \ldots, \underline{f}_n(x_1,\ldots,x_{n-1})\leq x_n = \overline{g}_n(x_1,\ldots,x_{n-1})\}
    \end{align*} where equality holds if $f_i<g_i$ for all $i\in[n]$.
    \item If $C \subseteq Q_n$ is $m$-extendable, then $C \in \mathcal{K}_{m,n}^B$ for each $B>B_C^m$.
\end{enumerate}
\end{corollary}

Lemma \ref{lem:nottoofar} implies that definable sets with the extension property are 'dense' among definable subsets of $Q_n$:

\begin{corollary}\label{cor:definablenottoofar}
Let $A\subseteq Q_n$ be a definable set. There exists $\ell = \ell(A,m) \in \N$ such that for every $\eps>0$, there is some definable set $A^\eps\subseteq A$ such that $A^\eps\in\mathcal{K}_{m,n}^{B,\ell}$ for some $B=B(\eps)>0$ and 
$$\lambda^n(A \setminus A^\eps) \leq \eps.$$
\end{corollary}
\begin{proof}
By Lemma \ref{lem:signature}, $A$ is $(m,\ell)$-traceable for some $\ell \in \N$. Applying Lemma \ref{lem:nottoofar} and following the definable construction of $A^\eps$ in its proof yields the claim.
\end{proof}
\begin{remark}
By \cite{van1998tame}, Chapter 4, Corollary (1.10), it follows that $\lambda^n(\cl(A) \setminus \interior(A))=0$. Hence internally approximating $\interior(A)$ using boxes, the simplest traceable sets with fixed regularity $B\leq 1/2$, yields that $A$ can be approximated by finite unions of fixed regularity sets from within. But for better approximation accuracy $\eps$, the number $\ell(\eps)$ of boxes might depend on the desired accuracy. Corollary \ref{cor:definablenottoofar} complements this approximation result for bounded definable sets, fixing the number $\ell$ of traceable sets needed for every accuracy $\eps$, while the regularity $B(\eps)$ may depend on $\eps$.
\end{remark}

\section{ReLU neural network approximation of traceable sets}\label{sec:approx}

With the function classes studied in \cite{petersen2018optimal}, we build approximating ReLU neural networks for characteristic functions of sets belonging to $\mathcal{K}_{m,n}^{B,\ell}$, keeping track of the architecture and the approximation error. We let $H_i: \R^n \to \R$, $H_i := \chi_{\R^{i-1} \times [0,\infty) \times \R^{n-i}}$ for $i\in[n]$.

\begin{definition}
For $B>0$, define the set of functions
$$\mathcal{F}_{m,n}^B:=\{ f \in C^m(Q_n):\norm{f}_{W^{m,\infty}(Q_n)} \leq B\}.$$
Further, for $n \geq 2$, let the \emph{horizon functions} with respect to $\mathcal{F}_{m,n}^B$ be given by
$$\mathcal{HF}_{m,n}^B:=\mathcal{HF}_{m,n}^{B,+} \cup \mathcal{HF}_{m,n}^{B,-}$$
where $$\mathcal{HF}_{m,n}^{B,+}:=\bigcup_{i=1}^n\{f\in L^{\infty}(Q_n):f(x)=H_i(x-e_i\gamma(x^{<i})):\gamma \in \mathcal{F}_{m,i-1}^B\}$$
and
$$\mathcal{HF}_{m,n}^{B,-}:=\bigcup_{i=1}^n\{f\in L^{\infty}(Q_n):f(x)=H_i(e_i\gamma(x^{<i})-x):\gamma \in \mathcal{F}_{m,i-1}^B\},$$
abbreviating $x^{< i} := (x_1,\ldots,x_{i-1})$.
\end{definition}

Let $n\geq 1$. For every $f \in \mathcal{F}_{m,n}^B$, we have that $-f \in \mathcal{F}_{m,n}^B$ and $\norm{\partial^\alpha f}_{L^\infty(Q_n)} \leq B$ for $\alpha \in \N^n,\abs{\alpha} \leq m$. 
So for $\abs{\alpha} = m-1$ we get
$$\abs{\partial^{\alpha} f(x) - \partial^{\alpha} f(y)} \leq \norm{(\sum_{i=1}^n (\partial^{\alpha+e_i} f)^2)^{1/2}}_{L^\infty(Q_n)}\norm{x-y} \leq \sqrt{n}B\norm{x-y}.$$
Hence $\mathcal{F}_{m,n}^B \subseteq \mathcal{F}_{\beta,n,\sqrt{n}B}$ for $\beta = (m-1)+\sigma$ with $\sigma = 1$ as defined in \cite{petersen2018optimal}, Equation (3.1).
Thus all approximation results concerning $\mathcal{F}_{\beta,n,\sqrt{n}B}$ apply to $\mathcal{F}_{m,n}^B$. 
Identifying every $f \in \mathcal{F}_{m,i-1}^B$ with a function $\gamma_f:Q_{n-1} \to \R$, $(x_1,\ldots,x_{n-1}) \mapsto f(x_1,\ldots,x_{i-1})$ 
yields the inclusion $\mathcal{F}_{m,i-1}^B \subseteq \mathcal{F}_{m,n-1}^B$ for each $i\in[n]$. This follows because $\norm{\partial^\alpha \gamma_f}_{L^\infty(Q_{n-1})} = 0$ if $\alpha_j > 0$ for some $j \geq i$.

Using $\mathcal{F}_{m,i-1}^B \subseteq \mathcal{F}_{m,n-1}^B \subseteq \mathcal{F}_{\beta,n-1,\sqrt{n}B}$ for $i \in [n]$ yields that $\mathcal{HF}_{m,n}^{B,+} \subseteq \mathcal{HF}_{\beta,n,\sqrt{n}B}$ for $\beta = (m-1)+\sigma$ and $\sigma = 1$ as defined in \cite{petersen2018optimal}, Definition 3.3, but $\mathcal{HF}_{\beta,n,\sqrt{n}B}$ does not include $\mathcal{HF}_{m,n}^{B,-}$.
However, for every $i\in[n]$ and $\gamma \in \mathcal{F}_{m,i-1}^B$, there holds 
$$H_i(e_i\gamma(x^{<i})-x) - (1 - H_i(x-e_i\gamma(x^{<i}))) = \chi_{\{x_i=\gamma(x^{<i})\}}(x),$$
so $x \mapsto H_i(e_i\gamma(x^{<i})-x)$ and $x \mapsto 1 - H_i(x-e_i\gamma(x^{<i}))$ agree in $L^p(Q_n)$ because $\{x \in Q_n: x_i=\gamma(x^{<i})\}$ is a null-set.
The $L^p$-approximation capabilities of  $\mathcal{HF}_{m,n}^{B,+}$ by ReLU neural networks directly translate to $1 - \mathcal{HF}_{m,n}^{B,+}$ and hence to $\mathcal{HF}_{m,n}^{B,-}$. This follows by observing that for a ReLU neural network $\Phi$, the function $x \mapsto 1-\Phi(x)$ is again the realization of a ReLU neural network $1-\Phi$ with the same architecture as $\Phi$. Also, $\depth(1-\Phi) = \depth(\Phi)$, $\size(1-\Phi) \leq \size(\Phi)+n$, and $\norm{1-\Phi}_{\max} \leq \norm{\Phi}_{\max}+1$.
Indeed, the ReLU neural network $1-\Phi$ emerges from $\Phi$ by multiplying all weights of the output layer with $-1$ and afterwards adding $+1$ to the entries of the output layer bias.

Hence we get the following version of \cite{petersen2018optimal}, Lemma 3.4, using that the asymptotic rates are invariant under replacing $B$ by $\sqrt{n}B$ modulo changing constants. 

\begin{theorem}\label{thm:optimalapprox_3.4}
For $B>0$, $n \geq 2$, there exist constants $c=c(n,m,B,p)>0$ and $s=s(n,m,B,p)>0$ such that for every $0<\eps<1/2$ and every $h \in \mathcal{HF}_{m,n}^B$, there is a ReLU neural network $\Phi_\eps^h: \R^n \to [0,1]$ with 
\begin{enumerate}[label=(\roman*)]
    \item $\norm{\Phi_\eps^h}_{\max} \leq \eps^{-s}$,
    \item $\depth(\Phi_\eps^h) \leq (2+\lceil \log_2(m)\rceil)(14+2m/n)$,
    \item $\size(\Phi_\eps^h) \leq c\cdot \eps^{-p(n-1)/m}$,
\end{enumerate}
such that
$$\norm{\Phi_\eps^h-h}_{L^p(Q_n)} < \eps.$$
\end{theorem}

\emph{Let $n\geq 2,B>0,0<\varepsilon<1/2,$ and $A \in \mathcal{K}_{m,n}^B$ with fundament $A_0,\ldots,A_n$ and signature datum $(\{f_i,g_i\}_{i\in[n]})$}. 
We aim to apply Theorem \ref{thm:optimalapprox_3.4} to approximate $\chi_A$.
For each $i\in[n]$, take $C^m$ maps $F_i,G_i:Q_{i-1} \to \R$ such that $F_i|_{A_{i-1}} = f_i$ and $G_i|_{A_{i-1}} = g_i$ with $\norm{F_i}_{W^{m,\infty}(Q_{i-1})},\norm{G_i}_{W^{m,\infty}(Q_{i-1})} \leq B$ so that $F_i,G_i \in \mathcal{F}_{m,i-1}^B$. 
Then the \emph{horizon functions $(\{\mathrm{F}_i,\mathrm{G}_i\}_{i\in[n]})$ of $A$ with respect to $(\{F_i,G_i\}_{i\in[n]})$} are defined as $\mathrm{F}_i,\mathrm{G}_i:Q_n\to\{0,1 \}$ where
$$\mathrm{F}_i(x) := H_i(x_1,\ldots,x_{i-1},x_i-F_i(x_1,\ldots,x_{i-1}),x_{i+1},\ldots, x_n)\in \mathcal{HF}_{m,n}^{B,+}$$
and 
$$\mathrm{G}_i(x) := H_i(x_1,\ldots,x_{i-1},G_i(x_1,\ldots,x_{i-1})-x_i,x_{i+1},\ldots, x_n)\in \mathcal{HF}_{m,n}^{B,-}$$
for $i\in[n]$.

Note that $\cl(A) = \bigcap_{i=1}^n \left(\mathrm{F}_i^{-1}(1) \cap \mathrm{G}_i^{-1}(1)\right)$ by Lemma \ref{lem:closuredescription} and hence $\chi_{\cl(A)} = \min_{i\in[n]}\{\mathrm{F}_i,\mathrm{G}_i\}$, independent of the choice of $F_i$ and $G_i$. Corollary \ref{cor:frontier} implies that $\norm{\chi_{\cl(A)}-\chi_A}_{L^p(Q_n)} = \norm{\chi_{\fr(A)}}_{L^p(Q_n)} = 0$. 

Further, let  $\Omega \in \mathcal{K}_{m,n}^{B,\ell}$. Then $\Omega = A^1 \cup \cdots \cup A^l$ for some $l\leq \ell$, where $A^j \in \mathcal{K}_{m,n}^B$ has signature datum $(\{f^j_i,g^j_i\}_{i\in[n]})$. For $C^m$ maps $F^j_i,G^j_i:Q_{i-1} \to \R$ such that $F^j_i|_{A^j_{i-1}} = f^j_i$ and $G^j_i|_{A^j_{i-1}} = g^j_i$ with $\norm{F^j_i}_{W^{m,\infty}(Q_{i-1})},\norm{G^j_i}_{W^{m,\infty}(Q_{i-1})} \leq B$, we get horizon functions $(\{\mathrm{F}^j_i,\mathrm{G}^j_i\}_{i\in[n]})$ of $A^j$ with respect to $(\{F^j_i,G^j_i\}_{i\in[n]})$ and we can express 
$$\cl(\Omega) = \cl(A^1) \cup \cdots \cup \cl(A^l) = \bigcup_{j=1}^l\bigcap_{i=1}^n \left((\mathrm{F}^j_i)^{-1}(1) \cap (\mathrm{G}^j_i)^{-1}(1)\right),$$
so $\chi_{\cl(\Omega)}= \max_{j\in[l]} \chi_{\cl(A^j)} = \max_{j\in[l]}\min_{i\in[n]}\{\mathrm{F}^j_i,\mathrm{G}^j_i\}$. Here, $\max_{j \in \emptyset} \chi_{A^j} := 0$. We again have 
$$\norm{\chi_{\cl(\Omega)}-\chi_\Omega}_{L^p(Q_n)} \leq \sum_{j=1}^l \norm{\chi_{\cl(A^j)}-\chi_{A^j}}_{L^p(Q_n)} = 0.$$
It remains to approximate $\chi_{\cl(\Omega)}$ by ReLU neural networks.

By Theorem \ref{thm:optimalapprox_3.4}, there are constants $c=c(n,m,B,p)>0$ and $s=s(n,m,B,p)>0$ independent of $A$ such that there are ReLU neural networks $\Phi_\eps^{\mathrm{F}_i}, \Phi_\eps^{\mathrm{G}_i}:\R^n\to[0,1]$ satisfying
\begin{enumerate}[label=(\roman*)]
    \item  $\norm{\Phi_\eps^{\mathrm{F}_i}}_{\max}, \norm{\Phi_\eps^{\mathrm{G}_i}}_{\max} \leq \eps^{-s}$,
    \item $\size(\Phi_\eps^{\mathrm{F}_i}), \size(\Phi_\eps^{\mathrm{G}_i}) \leq c\cdot \eps^{-p(n-1)/m}$,
    \item $\depth(\Phi_\eps^{\mathrm{F}_i}), \depth(\Phi_\eps^{\mathrm{G}_i}) \leq (2+\lceil \log_2(m)\rceil)(14+2m/n)$,
    \item $\norm{\Phi_\eps^{\mathrm{F}_i} - \mathrm{F}_i}_{L^p(Q_n)},\norm{\Phi_\eps^{\mathrm{G}_i} - \mathrm{G}_i}_{L^p(Q_n)} < \eps$.
\end{enumerate}
We list some lemmas from \cite{petersen2025mathematicaltheorydeeplearning} that assist in building and controlling a ReLU neural network with realization
$$\R^n \to \R, \quad x \mapsto \max_{j\in[l]}\min_{i\in[n]}\{h_{i,j}(x)\}$$
for maps $h_{i,j}:\R^n\to\R$.
The first result concerns parallelization with shared input \cite{petersen2025mathematicaltheorydeeplearning}, Section 5.1.3.

\begin{lemma}\label{lem:5.3}
For $\ell \in \N^*$, let $\Phi_1, \ldots ,\Phi_\ell$ be ReLU neural networks with realizations $\R^n \to \R$. Then there is a ReLU neural network $(\Phi_1,\ldots,\Phi_\ell)$ with realization $\R^n \to \R^\ell$ satisfying
$$(\Phi_1,\ldots,\Phi_\ell)(x) =(\Phi_1(x),\ldots,\Phi_\ell(x))$$
such that
\begin{enumerate}[label=(\roman*)]
    \item $\depth((\Phi_1,\ldots,\Phi_\ell)) = \max_{i\in[\ell]} \depth(\Phi_i)$,
    \item $\width((\Phi_1,\ldots,\Phi_\ell)) \leq 2\ell \max_{i\in[\ell]} \width(\Phi_i)$,
    \item $\size((\Phi_1,\ldots,\Phi_\ell)) \leq 2\ell\max_{i\in[\ell]}\size(\Phi_i)+2\ell\max_{i\in[\ell]}\depth(\Phi_i)$.
\end{enumerate}
\end{lemma}

Secondly, we recall the representation of maxima and minima by ReLU neural networks \cite{petersen2025mathematicaltheorydeeplearning}, Lemma 5.11.

\begin{lemma}\label{lem:5.11}
Let $n \geq 2$. Then there exist ReLU neural networks $\Phi^{\min}_n,\Phi^{\max}_n$ with realizations $\R^n \to \R$ satisfying
$$\Phi^{\min}_n(x_1, \ldots, x_n) = \min \{ x_1, \ldots, x_n \}, \quad \Phi^{\max}_n(x_1, \ldots, x_n) = \max \{ x_1, \ldots, x_n \}$$
such that
\begin{enumerate}[label=(\roman*)]
    \item $\depth(\Phi^{\min}_n),\depth(\Phi^{\max}_n) \leq \lceil \log_2(n) \rceil$,
    \item $\width(\Phi^{\min}_n),\width(\Phi^{\max}_n) \leq 3n$,
    \item $\size(\Phi^{\min}_n),\size(\Phi^{\max}_n) \leq 16n$.
\end{enumerate}
\end{lemma}

Lastly, we consider compositions of ReLU neural networks \cite{petersen2025mathematicaltheorydeeplearning}, Lemma 5.2.

\begin{lemma}\label{lem:5.2}
Let $\ell \in \N^*$ and $\Phi_1, \Phi_2$ be ReLU neural networks with realizations $\R^n \to \R^\ell, \R^\ell \to \R$. Then there is a ReLU neural network $\Phi_2 \bullet \Phi_1$ with realization $\R^n \to \R$ satisfying 
$$\Phi_2 \bullet \Phi_1(x) = \Phi_2(\Phi_1(x))$$
such that 
\begin{enumerate}[label=(\roman*)]
    \item $\depth(\Phi_2 \bullet \Phi_1) = \depth(\Phi_1) + \depth(\Phi_2)+1$,
    \item $\width(\Phi_2 \bullet \Phi_1) \leq 2\max \{\width(\Phi_1),\width(\Phi_2)\}$,
    \item $\size(\Phi_2 \bullet \Phi_1) \leq 2\size(\Phi_1) + 2\size(\Phi_2)$.
\end{enumerate}
\end{lemma}

Using Lemmas \ref{lem:5.3}, \ref{lem:5.11}, and \ref{lem:5.2}, we consider the ReLU neural network 
$$\Phi_\eps^A := \Phi_{2n}^{\min} \bullet (\Phi_\eps^{\mathrm{F}_1},\ldots,\Phi_\eps^{\mathrm{F}_n},\Phi_\eps^{\mathrm{G}_1}, \ldots, \Phi_\eps^{\mathrm{G}_n})$$
as the $L^p$ approximation of $\chi_{\cl(A)}$.
Observe that $\Phi_\eps^A(x) = \min_{j\in[n]}\{\Phi_\eps^{\mathrm{F}_j}(x),\Phi_\eps^{\mathrm{G}_j}(x)\}$ and that
\begin{align*}
\depth(\Phi_\eps^A) & \leq 1 + \depth(\Phi_{2n}^{\min} ) + \depth((\Phi_\eps^{\mathrm{F}_1},\ldots,\Phi_\eps^{\mathrm{F}_n},\Phi_\eps^{\mathrm{G}_1}, \ldots, \Phi_\eps^{\mathrm{G}_n})) \\ 
& \leq 1+\lceil\log_2(2n)\rceil+\max_{i\in[n]} \{\depth(\Phi_\eps^{\mathrm{F}_i}),\depth(\Phi_\eps^{\mathrm{G}_i})\} \\ 
& \leq 1+\lceil \log_2(2n)\rceil+(2+\lceil \log_2(m)\rceil)(14+2m/n)
\end{align*}
is bounded independently of $\eps$. Moreover,
\begin{align*}
\size(\Phi_\eps^A) & \leq 2\size(\Phi_{2n}^{\min}) + 2\size((\Phi_\eps^{\mathrm{F}_1},\ldots,\Phi_\eps^{\mathrm{F}_n},\Phi_\eps^{\mathrm{G}_1}, \ldots, \Phi_\eps^{\mathrm{G}_n})) \\
& \leq 64n + 8n \max_{i\in[n]}\{\size(\Phi_\eps^{\mathrm{F}_i}),\size(\Phi_\eps^{\mathrm{G}_i})\}+8n\max_{i\in[n]}\{\depth(\Phi_\eps^{\mathrm{F}_i}),\depth(\Phi_\eps^{\mathrm{G}_i})\} \\
& \leq 64n+ 8nc\eps^{-p(n-1)/m} + 8n(2+\lceil \log_2(m)\rceil)(14+2m/n) \\
& \leq \tilde{c}\cdot\eps^{-p(n-1)/m}
\end{align*}
where the constant $\tilde{c} = \tilde{c}(n,m,B,p) > 0$ is independent of $\eps$.

Also, $\norm{\Phi_\eps^A}_{\max} \leq \eps^{-\tilde{s}}$ where $\tilde{s} = \tilde{s}(n,m,B,p) > 0$ is a constant independent of $\eps$ as $\norm{\Phi_{2n}^{\min}}_{\max}$ is bounded by a constant, and the weights for the ReLU parallelization and composition networks do not add additional dependence on $\eps$ themselves. This step uses $\eps < 1/2$.

For the error estimation between $\chi_{\cl(A)}$ and $\Phi_\eps^A$ in $L^p$, we use the following lemma.

\begin{lemma}\label{lem:minmax}
For $\ell \in \N^*$ and $\alpha_1,\ldots,\alpha_\ell,\beta_1,\ldots,\beta_\ell \in L^p(\Omega)$ where $\Omega \subseteq \R^n$ is measurable, we have that 
$$\max\{\norm{\min_{i\in[\ell]} \alpha_i - \min_{i\in[\ell]} \beta_i}_{L^p(\Omega)}, \norm{\max_{i\in[\ell]} \alpha_i - \max_{i\in[\ell]} \beta_i}_{L^p(\Omega)}\} \leq \sum_{i=1}^\ell \norm{\alpha_i-\beta_i}_{L^p(\Omega)}.$$
\end{lemma}

\begin{proof}
Let $x \in \Omega$ and $j,k \in [\ell]$ such that $\alpha_j(x) = \min_{i\in[\ell]} \alpha_i(x)$ and $\beta_k(x) = \min_{i\in[\ell]} \beta_i(x)$. Then
$$\min_{i\in[\ell]} \alpha_i(x) - \min_{i\in[\ell]} \beta_i(x) = \alpha_j(x) - \beta_k(x) \leq \alpha_k(x) - \beta_k(x) \leq \max_{i\in[\ell]}\abs{\alpha_i(x) - \beta_i(x)}$$
and
$$\min_{i\in[\ell]} \beta_i(x) - \min_{i\in[\ell]} \alpha_i(x) = \beta_k(x) - \alpha_j(x) \leq \beta_j(x) - \alpha_j(x) \leq \max_{i\in[\ell]}\abs{\beta_i(x) - \alpha_i(x)}$$
almost everywhere, so $\abs{\min_{i\in[\ell]} \alpha_i(x) - \min_{i\in[\ell]} \beta_i(x)} \leq \max_{i\in[\ell]}\abs{\alpha_i(x) - \beta_i(x)}$ almost everywhere. 
With the estimate $$\norm{\max_{i\in[\ell]}\abs{\beta_i - \alpha_i}}_{L^p(\Omega)} \leq \norm{\sum_{i=1}^\ell \abs{\beta_i - \alpha_i}}_{L^p(\Omega)},$$
we get
$$\norm{\min_{i\in[\ell]} \alpha_i - \min_{i\in[\ell]} \beta_i}_{L^p(\Omega)} \leq \norm{\sum_{i=1}^\ell\abs{\beta_i - \alpha_i}}_{L^p(\Omega)} \leq \sum_{i=1}^\ell \norm{\beta_i - \alpha_i}_{L^p(\Omega)}$$
using Minkowski's inequality. 
Similarly, taking  $j,k \in [\ell]$ such that $\alpha_j(x) = \max_{i\in[\ell]} \alpha_i(x)$ and $\beta_k(x) = \max_{i\in[\ell]} \beta_i(x)$ yields that
$$\max_{i\in[\ell]} \alpha_i(x) - \max_{i\in[\ell]} \beta_i(x) = \alpha_j(x) - \beta_k(x) \leq \alpha_j(x) - \beta_j(x) \leq \max_{i\in[\ell]}\abs{\alpha_i(x) - \beta_i(x)}$$
and
$$\max_{i\in[\ell]} \beta_i(x) - \max_{i\in[\ell]} \alpha_i(x) = \beta_k(x) - \alpha_j(x) \leq \beta_k(x) - \alpha_k(x) \leq \max_{i\in[\ell]}\abs{\beta_i(x) - \alpha_i(x)}$$
almost everywhere, and the same reasoning as above shows the lemma.
\end{proof}

Letting $\alpha_1:=\mathrm{F}_1,\ldots, \alpha_n:=\mathrm{F}_n,\alpha_{n+1}:=\mathrm{G}_1,\ldots, \alpha_{2n}:=\mathrm{G}_n$ and $\beta_1:=\Phi_\eps^{\mathrm{F}_1},\ldots,\beta_n:=\Phi_\eps^{\mathrm{F}_n},\beta_{n+1}:=\Phi_\eps^{\mathrm{G}_1},\ldots,\beta_{2n}:=\Phi_\eps^{\mathrm{G}_n}$ in Lemma \ref{lem:minmax} gives
$$\norm{\chi_{\cl(A)} - \Phi_\eps^A}_{L^p(Q_n)} \leq \sum_{i=1}^n (\norm{\Phi_\eps^{\mathrm{F}_i}-\mathrm{F}_i}_{L^p(Q_n)} + \norm{\Phi_\eps^{\mathrm{G}_i}-\mathrm{G}_i}_{L^p(Q_n)}) < 2n\eps.$$
Hence
\begin{align*}
\norm{\chi_A-\Phi_{\eps/2n}^A}_{L^p(Q_n)} & \leq \norm{\chi_{\cl(A)}-\Phi_{\eps/2n}^A}_{L^p(Q_n)} + \norm{\chi_A-\chi_{\cl(A)}}_{L^p(Q_n)}  \\
& = \norm{\chi_{\cl(A)}-\Phi_{\eps/2n}^A}_{L^p(Q_n)} + 0 < \eps.    
\end{align*}
Note that $\depth(\Phi_{\eps/2n}^A)$, $\size(\Phi_{\eps/2n}^A)$, and $\norm{\Phi_{\eps/2n}^A}_{\max}$ satisfy the same asymptotic bounds as the corresponding values for $\Phi_{\eps}^A$ by adjusting the constants.
The following theorem summarizes our findings. 

\begin{theorem}\label{thm:singletracable}
For $B>0$, $n \geq 2$, there are constants $c=c(n,m,B,p)>0$ and $s=s(n,m,B,p) > 0$ such that for every $A \in \mathcal{K}_{m,n}^B$ and $0<\eps < 1/2$, there is a ReLU neural network $\Phi_\eps^A:\R^n \to [0,1]$ satisfying
\begin{enumerate}[label=(\roman*)]
    \item $\norm{\Phi_\eps^A}_{\max} \leq \eps^{-s}$,
    \item $\depth(\Phi_\eps^A) \leq 1+\lceil \log_2(2n)\rceil+(2+\lceil \log_2(m)\rceil)(14+2m/n)$,
    \item $\size(\Phi_\eps^A) \leq c\cdot\eps^{-p(n-1)/m}$,
\end{enumerate}
with 
$$\norm{\chi_A-\Phi_\eps^A}_{L^p(Q_n)} < \eps.$$
\end{theorem}

Next, we approximate the more general class $\mathcal{K}_{m,n}^{B,\ell}$ by ReLU neural networks. 

\emph{Let $n\geq 2, B>0, \ell \in \N,$ and $\Omega \in \mathcal{K}_{m,n}^{B,\ell}$}. 
Then $\Omega = A^1 \cup \cdots \cup A^l$ for some $l\leq \ell$, where $A^j \in \mathcal{K}_{m,n}^B$.
Applying Theorem \ref{thm:singletracable} to each $A^j$ for $j\in[l]$ yields ReLU neural networks $\Phi^j_\eps: \R^n \to [0,1]$ satisfying
\begin{enumerate}[label=(\roman*)]
    \item $\norm{\Phi_\eps^j}_{\max} \leq \eps^{-s}$,
    \item $\depth(\Phi_\eps^j) \leq \lceil \log_2(n)\rceil+(3+\lceil \log_2(m)\rceil)(14+2m/n)$,
    \item $\size(\Phi_\eps^j) \leq c\cdot\eps^{-p(n-1)/m}$,
    \item $\norm{\chi_{A^j}-\Phi_\eps^j}_{L^p(Q_n)} < \eps,$
\end{enumerate}
with constants $c$ and $s$ independent of $A^j$. Further, for $j\in\{l+1,\ldots,\ell\}$, let $\Phi^j_\eps$ be the ReLU neural network with architecture $(n,1,1)$ and all weights set to zero, so that $\Phi^j_\eps(x) = 0$ for all $x \in \R^n$.

To approximate $\chi_{\cl(\Omega)} = \max_{j\in[l]} \chi_{\cl(A^j)}$, we define the ReLU neural network
\begin{equation*}
\Phi^\Omega_\eps :=
\begin{cases*}
\Phi^{\max}_\ell \bullet (\Phi_\eps^1,\ldots,\Phi_\eps^\ell), \quad \ell \geq 2 \\
\Phi_\eps^1, \quad \ell = 1 \\
0, \quad \ell = 0.
\end{cases*}
\end{equation*}
If $\ell = 0$, we take the approximating ReLU neural network of $\Omega = \emptyset$ with architecture $(n,1,1)$ and all weights set to zero. In case $\ell = 1$ and $\Omega \neq \emptyset$, we get $\Omega \in \mathcal{K}_{m,n}^B$, which is covered by Theorem \ref{thm:singletracable}. So we assume $\ell \geq 2$. Then $\Phi^\Omega_\eps(x) = \max \{\Phi_\eps^1(x), \ldots, \Phi_\eps^l(x)\}$. Defining $A^{l+1}=\cdots=A^\ell= \emptyset$,  $\Phi_\eps^{l+1} = \cdots = \Phi_\eps^\ell = 0$ and $\max_{j\in\emptyset} \Phi_\eps^j = 0$, using Lemma \ref{lem:minmax} yields
$$\norm{\chi_{\Omega} - \Phi^\Omega_\eps}_{L^p(Q_n)} \leq \sum_{j=1}^\ell \norm{\chi_{A^j} - \Phi^j_\eps}_{L^p(Q_n)} < \ell\eps.$$
It remains to study the architecture of $\Phi^\Omega_\eps$. We observe that
\begin{align*}
\depth(\Phi^\Omega_\eps) & = 1 + \depth(\Phi_\ell^{\max} ) + \depth((\Phi_\eps^1,\ldots,\Phi_\eps^\ell)) \\ 
& \leq 1+\lceil\log_2(\ell)\rceil+\max_{j\in[\ell]} \depth(\Phi^j_\eps) \\ 
& \leq 1+\lceil \log_2(\ell)\rceil+1+\lceil \log_2(2n)\rceil+(2+\lceil \log_2(m)\rceil)(14+2m/n) \\
& \leq \lceil \log_2(\ell n)\rceil +(3+\lceil \log_2(m)\rceil)(14+2m/n)
\end{align*}
is bounded independently of $\eps$ and
\begin{align*}
\size(\Phi^\Omega_\eps) & \leq 2\size(\Phi_\ell^{\max}) + 2\size((\Phi_\eps^1,\ldots,\Phi_\eps^\ell)) \\
& \leq 32\ell + 4\ell\max_{j\in[\ell]}\size(\Phi^j_\eps)+4\ell\max_{j\in[\ell]}\depth(\Phi_\eps^j) \\
& \leq 32\ell+ 4\ell c\eps^{-p(n-1)/m} + 4\ell (1+\lceil \log_2(2n)\rceil+(2+\lceil \log_2(m)\rceil)(14+2m/n)) \\
& \leq \tilde{c} \cdot \eps^{-p(n-1)/m}
\end{align*}
where the constant $\tilde{c} = \tilde{c}(\ell,m,n,B,p)$ is independent of $\eps$.

Also, $\norm{\Phi^\Omega_\eps}_{\max} \leq \eps^{-\tilde{s}}$ where $\tilde{s}(\ell,m,n,B,p)$ is independent of $\eps$ as $\norm{\Phi_\ell^{\max}}_{\max}$ is fixed and
the ReLU neural networks for parallelization and composition do not add additional dependence on $\eps$ themselves. This step uses $\eps < 1/2$.
Note that $\depth(\Phi_{\eps/\ell}^\Omega)$, $\size(\Phi_{\eps/\ell}^\Omega)$, and $\norm{\Phi_{\eps/\ell}^\Omega}_{\max}$ satisfy the same asymptotic bounds as the corresponding values for $\Phi_{\eps}^\Omega$ by adjusting the constants.

Hence we established Theorem \ref{thm:mainapprox}.

\subsection{Approximating piecewise smooth maps}\label{sec:piecewisesmooth}

\emph{In this section, $n\geq 2$}. Let $f:Q_n\to\R$ be a definable map. By Theorem \ref{thm:omincell}, for each $m'\in\N^*$, there is a $C^{m'}$-cell decomposition $\mathcal{C}$ of $Q_n$ so that $f|_C$ is $C^{m'}$ for every $C\in\mathcal{C}$.
Motivated by this fact and by \cite{petersen2018optimal}, we establish ReLU neural network approximation of piecewise smooth functions whose smoothness domains are given by finitely many pairwise disjoint sets in $\mathcal{K}_{m,n}^B$ for some $B>0$, with smoothness of degree $\lceil mn/(pn-p)\rceil$ on each piece.

The tool we use to achieve approximation results in this setting is the ReLU approximation of the function class $\mathcal{F}_{m,n}^B \subseteq \mathcal{F}_{m,n,\sqrt{n}B}$, which can be inferred from \cite{petersen2018optimal}, Theorem 3.1, when absorbing the factor $\sqrt{n}$ into the constants.

\begin{theorem}\label{thm:optimalapprox_3.1}
For $n\geq 2$ and $B > 0$, there exist constants $c=c(m,n,B)>0$ and $s=s(m,n,B,p)>0$ such that for every $f \in \mathcal{F}_{m,n}^B$ and $0<\eps < 1/2$, there is a ReLU neural network $\Phi_\eps^f:\R^n \to \R$ with
\begin{enumerate}[label=(\roman*)]
    \item $\norm{\Phi_\eps^f}_{L^\infty(Q_n)} \leq \lceil \sqrt{n}B\rceil$,
    \item $\norm{\Phi_\eps^f}_{\max} \leq \eps^{-s}$,
    \item $\size(\Phi_\eps^f) \leq c \cdot \eps^{-n/m}$,
    \item $\depth(\Phi_\eps^f) \leq (2+\lceil \log_2(m)\rceil)(11+m/n)$,
\end{enumerate}
and $$\norm{f - \Phi_\eps^f}_{L^p(Q_n)} < \eps.$$
\end{theorem}

Fix $B>0$ and $\ell\in\N$. Balancing the asymptotic rates for the number of non-zero weights of ReLU neural networks approximating the class $\mathcal{K}_{m,n}^B$  (that is, $\eps^{-p(n-1)/m}$) and the function class $\mathcal{F}_{m,n}^B$ (that is, $\eps^{-n/m}$) leads to the consideration of
$$\mathcal{G}^{B,\ell}_{m,n,p}:=\{\sum_{i=1}^\ell g_i \cdot \chi_{A^i}: \quad g_i \in \mathcal{F}_{\lceil mn/(pn-p)\rceil,n}^B, \,\,A^i \in \mathcal{K}_{m,n}^B,\,\, A^i\cap A^j = \emptyset, i \neq j \}.$$
First, we recall that the multiplication operator on bounded domains can be approximated by ReLU neural networks with a fixed number of layers \cite{petersen2018optimal}, Lemma A.3.

\begin{lemma}\label{lem:multapprox}
Let $\theta > 0$. For every $L\in\N^*,L>(2\theta)^{-1}$ and $M \geq 1$, there are constants $c=c(L,M,\theta)\in\N^*$ and $s=s(M)\in \N^*$ such that for every $0 < \eps < 1/2$, there is a ReLU neural network $\tilde{\times}_\eps: \R^2 \to \R$ satisfying
\begin{enumerate}[label=(\roman*)]
    \item $\abs{\tilde{\times}_\eps(x,y) -x\cdot y} \leq \eps$ and $\tilde{\times}_\eps(0,y) = \tilde{\times}_\eps(x,0) = 0$ for all $(x,y) \in [-M,M]^2$,
    \item $\norm{\tilde{\times}_\eps}_{\max} \leq \eps^{-s}$,
    \item $\size(\tilde{\times}_\eps) \leq c\cdot \eps^{-\theta}$,
    \item $\depth(\tilde{\times}_\eps) \leq 2L+7$.
\end{enumerate}
\end{lemma}

With Lemma \ref{lem:multapprox}, we first establish approximation of the simpler function class 
$$\mathcal{E}^B_{m,n,p}:= \{g\cdot \chi_A: \quad g \in \mathcal{F}_{\lceil mn/(pn-p)\rceil,n}^B,A \in \mathcal{K}_{m,n}^B\}$$
by ReLU neural networks.

\begin{theorem}\label{thm:smoothsingle}
For $n\geq 2$ and $B > 0$, there exist constants $c=c(m,n,B,p)>0$ and $s=s(m,n,B,p)>0$ such that for every $f \in \mathcal{E}^B_{m,n,p}$ and $0<\eps < 1/2$, there is a ReLU neural network $\Phi_\eps^f:\R^n \to \R$ with
\begin{enumerate}[label=(\roman*)]
    \item $\norm{\Phi_\eps^f}_{\max} \leq \eps^{-s}$,
    \item $\size(\Phi_\eps^f) \leq c \cdot \eps^{-p(n-1)/m}$,
    \item $\depth(\Phi_\eps^f) \leq (5+\lceil\log_2(mn)\rceil)(14+2m/n+m/(pn-p))$,
\end{enumerate}
and $$\norm{f - \Phi_\eps^f}_{L^p(Q_n)} < \eps.$$
\end{theorem}

\begin{proof}
Let $f=g \cdot \chi_A$ for $g \in \mathcal{F}_{\lceil mn/(pn-p)\rceil,n}^B$ and $A \in \mathcal{K}_{m,n}^B$. Applying Theorem \ref{thm:singletracable} to $A$ and Theorem \ref{thm:optimalapprox_3.1} to $g$ yields constants $c(n,m,B,p)>0$ and $s(n,m,B,p)>0$ independent of $\eps$ and $A$ such that there are ReLU neural networks $\Phi^A_\eps$  and  $\Phi^g_\eps$ satisfying
\begin{enumerate}[label=(\roman*)]
    \item $\norm{\Phi_\eps^g}_{L^\infty(Q_n)} \leq \lceil \sqrt{n}B\rceil$, $\norm{\Phi_\eps^A}_{L^\infty(Q_n)} \leq 1$,
    \item $\norm{\Phi_\eps^g}_{\max},\norm{\Phi_\eps^A}_{\max} \leq \eps^{-s}$,
    \item $\size(\Phi_\eps^g), \size(\Phi_\eps^A) \leq c \cdot \eps^{-p(n-1)/m}$,
    \item $\depth(\Phi_\eps^g)\leq (2+\lceil \log_2(\lceil mn/(pn-p)\rceil)\rceil)(12+m/(pn-p))$,
    \item $\depth(\Phi_\eps^A) \leq \lceil \log_2(n)\rceil+(3+\lceil \log_2(m)\rceil)(14+2m/n)$,
    \item $\norm{g - \Phi_\eps^g}_{L^p(Q_n)}, \norm{\chi_A - \Phi_\eps^A}_{L^p(Q_n)} < \eps.$
\end{enumerate}

So $\depth(\Phi_\eps^g),\depth(\Phi_\eps^A) \leq (3+\lceil\log_2(mn)\rceil)(14+2m/n+m/(pn-p))$.
Since $\abs{\Phi_\eps^g} \leq \lceil \sqrt{n}B \rceil$ and $0\leq \Phi^A_\eps \leq 1$, multiplying $\Phi^A_\eps$  and  $\Phi^g_\eps$ can be approximated by the ReLU neural network $\tilde{\times}_\eps :\R^2 \to \R$ defined in Lemma \ref{lem:multapprox} with $M=\lceil \sqrt{n}B \rceil,\theta =p(n-1)/m,$ and $L=\lfloor(2\theta)^{-1}\rfloor+1$. Here, $\size(\tilde{\times}_\eps) \leq c \eps^{-p(n-1)/m}$ by possibly enlarging $c$. Define the approximating ReLU neural network by
$$\Phi_\eps^f := \tilde{\times}_\eps \bullet(\Phi_\eps^g,\Phi_\eps^A).$$
Then we get
\begin{align*}
\depth(\Phi_\eps^f) & \leq \depth(\tilde{\times}_\eps)+\depth((\Phi_\eps^g,\Phi_\eps^A))+1 \\
& \leq 2L + 8 +  (3+\lceil\log_2(mn)\rceil)(14+2m/n+m/(pn-p)) \\
& \leq m/(pn-p) + (4+\lceil\log_2(mn)\rceil)(14+2m/n+m/(pn-p)) \\
& \leq (5+\lceil\log_2(mn)\rceil)(14+2m/n+m/(pn-p))
\end{align*}
and
\begin{align*}
\size(\Phi_\eps^f) & \leq 2\size(\tilde{\times}_\eps)+2\size((\Phi_\eps^g,\Phi_\eps^A))\\ 
& \leq 2c\eps^{-\theta}+8c\eps^{-p(n-1)/m} 
 +8(3+\lceil\log_2(mn)\rceil)(14+2m/n+m/(pn-p)) \\
& \leq 10c \eps^{-p(n-1)/m} + (24+8\lceil\log_2(mn)\rceil)(14+2m/n+m/(pn-p))\\
& \leq \tilde{c}\eps^{-p(n-1)/m}
\end{align*}
for some constant $\tilde{c} = \tilde{c}(n,m,B,p) > 0$.  Also, $\norm{\Phi_\eps^f}_{\max} \leq \eps^{-\tilde{s}}$ for some $\tilde{s} = \tilde{s}(n,m,B,p) > 0$ as $\norm{\tilde{\times}_\eps}_{\max} \leq \eps^{-\hat{s}}$ for some $\hat{s}>0$ and the ReLU neural networks for parallelization as well as composition do not add additional dependence on $\eps$ themselves. This step uses $\eps < 1/2$.

To estimate $\norm{f-\Phi_\eps^f}_{L^p(Q_n)}$, we calculate
\begin{align*}
\norm{f-\Phi_\eps^f}_{L^p(Q_n)} 
& = \norm{g \cdot \chi_A-\tilde{\times}_\eps \bullet(\Phi_\eps^g,\Phi^A_\eps)}_{L^p(Q_n)} \\
& \leq  \norm{g \cdot \chi_A-\Phi_\eps^g\cdot\Phi^A_\eps}_{L^p(Q_n)} \\
& + \norm{\Phi_\eps^g\cdot \Phi^A_\eps - \tilde{\times}_\eps \bullet(\Phi_\eps^g,\Phi^A_\eps)}_{L^p(Q_n)} \\
& \leq \norm{g \cdot \chi_A-\Phi_\eps^g\cdot \chi_A}_{L^p(Q_n)} \\
& + \norm{\Phi_\eps^g\cdot \chi_A-\Phi_\eps^g\cdot\Phi^A_\eps}_{L^p(Q_n)}+\eps \\
&\leq \norm{g -\Phi_\eps^g}_{L^p(Q_n)} \\
& + \lceil \sqrt{n}B \rceil \norm{\chi_A-\Phi^A_\eps}_{L^p(Q_n)} + \eps \\
& < (2+\lceil \sqrt{n}B \rceil)\eps.
\end{align*}
Note that $\depth(\Phi_{\eps/(2+\lceil \sqrt{n}B \rceil)}^f)$, $\size(\Phi_{\eps/(2+\lceil \sqrt{n}B \rceil)}^f)$, and $\norm{\Phi_{\eps/(2+\lceil \sqrt{n}B \rceil)}^f}_{\max}$ satisfy the same asymptotic bounds as the corresponding values for $\Phi_{\eps}^f$ by adjusting the constants. Hence replacing $\eps$ by  $\eps/(2+\lceil \sqrt{n}B \rceil)$ yields the result.
\end{proof}

To approximate the unique element of $\mathcal{G}^{B,0}_{m,n,p}$, take the ReLU neural network with architecture $(n,1,1)$ and all weights set to zero.
For approximations of elements in $\mathcal{G}^{B,\ell}_{m,n,p}$ for $\ell \geq 1$, it remains to sum over the approximating ReLU neural networks of each single region.
The ReLU neural network expressing a sum of $\ell$ many ReLU neural networks is a special case of taking linear combinations \cite{petersen2025mathematicaltheorydeeplearning}, Lemma 5.4.

\begin{lemma}\label{lem:5.4}
Let $\ell \in \N^*$ and $\Phi_1,\ldots,\Phi_\ell:\R^n \to \R$ be ReLU neural networks. Then there is a ReLU neural network $\sum(\Phi_1,\ldots,\Phi_\ell): \R^n \to \R$ such that 
\begin{enumerate}[label=(\roman*)]
    \item $\depth(\sum(\Phi_1,\ldots,\Phi_\ell)) = \max_{i\in[\ell]}\depth(\Phi_i)$,
    \item $\width(\sum(\Phi_1,\ldots,\Phi_\ell)) \leq 2\ell \max_{i\in[\ell]} \width(\Phi_i)$,
    \item $\size(\sum(\Phi_1,\ldots,\Phi_\ell)) \leq 2\ell \max_{i\in[\ell]}\size(\Phi_i) + 2\ell\max_{i\in[\ell]}\depth(\Phi_i)$,
\end{enumerate}
with $\sum(\Phi_1,\ldots,\Phi_\ell)(x) =  \Phi_1(x) + \cdots + \Phi_\ell(x)$ for every $x\in\R^n$.
\end{lemma}

Our candidate for the approximation of $f := \sum_{i=1}^\ell g_i\chi_{A^i}\in \mathcal{G}_{m,n,p}^{B,\ell}$
is given by
$$\Phi_\eps^f := \sum(\Phi^{g_1\chi_{A^1}}_\eps,\ldots,\Phi^{g_\ell\chi_{A^\ell}}_\eps),$$ where $\Phi^{g_i\chi_{A^i}}_\eps$ are the approximating ReLU neural networks of $g_i\chi_{A^i}$ coming from Theorem \ref{thm:smoothsingle}. Estimating yields 
\begin{align*}
\norm{f-\Phi_\eps^f}_{L^p(Q_n)} & \leq \sum_{i=1}^\ell \norm{g_i\chi_{A^i}-\Phi^{g_i\chi_{A^i}}_\eps}_{L^p(Q_n)} \\
& \leq \ell \max_{i\in[\ell]} \norm{g_i\chi_{A^i}-\Phi^{g_i\chi_{A^i}}_\eps}_{L^p(Q_n)} < \ell\eps.
\end{align*}
It remains to study the architecture of $\Phi_\eps^f$. We have
\begin{align*}
\depth(\Phi_\eps^f) & = \max_{i\in[\ell]}\depth(\Phi^{g_i\chi_{A^i}}_\eps) \\
&\leq (5+\lceil\log_2(mn)\rceil)(14+2m/n+m/(pn-p))
\end{align*}
and
\begin{align*}
\size(\Phi_\eps^f) 
& \leq  2\ell \max_{i\in[\ell]}\size(\Phi^{g_i\chi_{A^i}}_\eps) + 2\ell\max_{i\in[\ell]}\depth(\Phi^{g_i\chi_{A^i}}_\eps) \\
& \leq 2\ell c\eps^{-p(n-1)/m} + 2\ell(5+\lceil\log_2(mn)\rceil)(14+2m/n+m/(pn-p)) \\
&\leq \tilde{c} \cdot \eps^{-p(n-1)/m}
\end{align*}
for some constant $\tilde{c} = \tilde{c}(n,m,B,p,\ell) > 0$. Also, we get that $\norm{\Phi_\eps^f}_{\max} \leq \eps^{-\tilde{s}}$ for some constant $\tilde{s} = \tilde{s}(n,m,B,p,\ell) > 0$ as the summation ReLU neural network does not depend on $\eps$. 
Note that $\depth(\Phi_{\eps/\ell}^f)$, $\size(\Phi_{\eps/\ell}^f)$, and $\norm{\Phi_{\eps/\ell}^f}_{\max}$ satisfy the same asymptotic bounds as the corresponding values for $\Phi_{\eps}^f$ by adjusting the constants.

Thus we proved the following approximation result when replacing $\eps$ by $\eps/\ell$ if $\ell \geq 1$ in the discussion above and otherwise using the zero ReLU neural network.

\begin{theorem}
For $n\geq 2$, $B > 0$ and $\ell \in \N$, there exist constants $c= c(n,m,B,p,\ell) > 0$ and $s=s(n,m,B,p,\ell)>0$ such that for every $f \in \mathcal{G}^{B,\ell}_{m,n,p}$ and $0<\eps < 1/2$, there is a ReLU neural network $\Phi_\eps^f: \R^n \to \R$ with
\begin{enumerate}[label=(\roman*)]
    \item $\norm{\Phi_\eps^f}_{\max} \leq \eps^{-s}$,
    \item $\size(\Phi_\eps^f) \leq c \cdot \eps^{-p(n-1)/m}$,
    \item $\depth(\Phi_\eps^f) \leq (5+\lceil\log_2(mn)\rceil)(14+2m/n+m/(pn-p))$,
\end{enumerate}
and $$\norm{f - \Phi_\eps^f}_{L^p(Q_n)} \leq \eps.$$
\end{theorem}

\section{Convergence rates for statistical learning of traceable sets}\label{sec:estim}

\emph{Let $n\geq 2,B>0,$ and $\ell\in\N$}. In this section, we derive upper bounds for the minimax error when learning sets from $\mathcal{K}_{m,n}^{B,\ell}$ via empirical risk minimization on $N\geq 1$ samples over a suitable class of ReLU neural networks. We follow the setting considered in \cite{petersen2021optimal}, except for working on $Q_n$ instead of $[0,1]^n$. In the following, let $c$ and $s$ be the positive constants from Theorem \ref{thm:mainapprox}.

For our binary classification task, suppose that we are given a labelled sample $S:= (x_i,y_i)_{i=1}^N$ where $x_i$ are sampled independently from the Lebesgue measure $\lambda^n$ on $Q_n$ and $y_i := \chi_\Omega(x_i)$ for some given unknown $\Omega \in \mathcal{K}_{m,n}^{B,\ell}$ where $\ell\in\N$ and $B >0$ are fixed. Based on the sample $S$, we want to predict $\chi_\Omega$ while keeping track of the estimation error depending on $N$.

For an upper bound on the minimax error $\eps_N(\mathcal{K}_{m,n}^{B,\ell})$ introduced in Section \ref{sec:mainresults}, we select for each sample $S \in \Lambda_N$ an empirical risk minimizer $\Phi_S$ with respect to the hinge loss $\{0,1\}\times[0,1] \to \R, (y,\hat y) \mapsto 1-(2y-1)(2\hat y-1)$ over the ReLU neural network class
$$\mathcal{F}_N:= \mathcal{NN}(n,L,cN^{p(n-1)/(m+pn-p)},N^{sm/(m+pn-p)})$$
with $L:=\lceil \log_2((\ell+1)n)\rceil+(3+\lceil \log_2(m)\rceil)(14+2m/n)$ and $p\geq 1$. Note that $\mathcal{F}_N$ is the image of the restricted Lipschitz continuous realization maps (\cite{petersen2025mathematicaltheorydeeplearning}, Proposition 13.1) applied to each of the finitely many compact components of the parameter set $\Theta_{n,L,cN^{p(n-1)/(m+pn-p)},N^{sm/(m+pn-p)}}$ encoding the weights of ReLU neural networks representing $\mathcal{F}_N$, as discussed in Section \ref{sec:nnconvention}. Hence an empirical risk minimizer exists.

Since an element $\Phi \in \mathcal{F}_N$ takes values in $[0,1]$, we need a decision rule for the prediction of the label. 
Taking the threshold $1/2$ yields the predictor $\chi_{[1/2, \infty)}\circ\Phi \in L^2(Q_n)$ that assigns the label $\hat{y}_i = 0$ if $\Phi(x_i)<1/2$ and the label $\hat{y}_i = 1$ if $\Phi(x_i)\geq 1/2$.

The hinge risk of $\Phi$ with respect to $\Omega$ is obtained by 
$$\mathcal{E}_{\Omega}(\Phi) := \mathbb{E}_{X \sim \lambda^n} [1-(2\chi_\Omega(X)-1)(2\Phi(X) -1)].$$
The label prediction is $\mathcal{A}(S) := \chi_{[1/2, \infty)}\circ\Phi_S$, where $\Phi_S$ satisfies
$$\Phi_S \in \argmin_{\Phi \in \mathcal{F}_N} 1-\frac{1}{N} \sum_{i=1}^N (2y_i-1)(2\Phi(x_i)-1)$$
for every $S \in \Lambda_N$.

The following lemma bounds the hinge risk of the approximating ReLU neural network of $\chi_\Omega$ obtained from Theorem \ref{thm:mainapprox}.

\begin{lemma}\label{lem:estim1}
Let $B>0$, $\ell \in \N$, $p\geq 1$, and $0<\eps<1/2$. For $\Omega \in \mathcal{K}_{m,n}^{B,\ell}$ and $\Phi_\eps^\Omega$ as in Theorem \ref{thm:mainapprox}, we have $$\mathcal{E}_{\Omega}(\Phi_\eps^\Omega) \leq 2\eps.$$
\end{lemma}
\begin{proof}
We calculate
\begin{align*}
\mathcal{E}_{\Omega}(\Phi_\eps^\Omega) & = \mathbb{E}_{X\sim \lambda^n}[1-(2\chi_\Omega(X)-1)(2\Phi_\eps^\Omega(X)-1)] \\
& = \int_{Q_n} (1-(2\chi_\Omega-1)(2\Phi_\eps^\Omega-1)) d\lambda^n \\
& = \int_\Omega (1-(2\Phi_\eps^\Omega-1)) d\lambda^n + \int_{Q_n \setminus \Omega} (1-(1-2\Phi_\eps^\Omega)) d\lambda^n \\
& = 2\int_\Omega (1-\Phi_\eps^\Omega) d\lambda^n + 2\int_{Q_n \setminus \Omega} \Phi_\eps^\Omega d\lambda^n\\
& = 2 \int_{Q_n} \abs{\chi_\Omega-\Phi_\eps^\Omega}d\lambda^n \\
& = 2 \norm{\chi_\Omega-\Phi_\eps^\Omega}_{L^1(Q_n)} \leq 2 \norm{\chi_\Omega-\Phi_\eps^\Omega}_{L^p(Q_n)} \leq 2\eps.
\end{align*}
The last line uses $\lambda^n(Q_n) = 1$ and invokes Hölder's inequality with $q = p/(p-1)$ for $p>1$ and $q=\infty$ for $p=1$.
\end{proof}

For sufficiently large $N$, applying Theorem \ref{thm:mainapprox} with $0<\eps := N^{-m/(m+pn-p)}<1/2$, we get some $\Phi_\eps^\Omega\in \mathcal{NN}(n,L,W_N,B_N)$, where
\begin{align*}
L&:=\lceil \log_2((\ell+1)n)\rceil+(3+\lceil \log_2(m)\rceil)(14+2m/n),\\
W_N&:=cN^{p(n-1)/(m+pn-p)},\\
B_N&:=N^{sm/(m+pn-p)}
\end{align*}
are non-decreasing in $N$ with constants  $c = c(n,m,\ell,B,p)>0$ and $s = s(n,m,\ell,B,p)>0$ independent of $\Omega\in\mathcal{K}_{m,n}^{B,\ell}$, such that
$$\norm{\chi_\Omega-\Phi_\eps^\Omega}_{L^p(Q_n)} < N^{-m/(m+pn-p)}.$$
In other words, there is some  $\Phi_N \in \mathcal{F}_N$ with $\norm{\chi_\Omega-\Phi_N}_{L^p(Q_n)} < N^{-m/(m+pn-p)}$.

The forthcoming statements include an important condition on the intrinsic structure of $\mathcal{F}_N$.
A \emph{$\delta$-net} of $\mathcal{F}_N$ is some finite $\mathcal{G}_N \subseteq \mathcal{F}_N$ such that
$$\mathcal{F}_N \subseteq \bigcup_{g\in\mathcal{G}_N}\{f\in L^\infty(Q_n): \norm{g-f}_{L^\infty(Q_n)} \leq \delta\}$$
and the \emph{covering number} of $\mathcal{F}_N$ is the minimal size of a $\delta$-net of $\mathcal{F}_N$. The \emph{covering entropy} $V_{\mathcal{F}_N,\norm{.}_{L^\infty}}(\delta)$ is the logarithm of the covering number of $\mathcal{F}_N$. We list the conditions from \cite{petersen2021optimal}, Theorem 5.6, that need to hold for $\mathcal{F}_N$ to prove our learning result (Theorem \ref{thm:statlearn}).

\begin{theorem}\label{thm:5.6optimal}
Let $\Omega \in \mathcal{K}_{m,n}^{B,\ell}$. If
\begin{enumerate}[label=(\roman*)]
    \item \label{enu:approxcond} there exists a positive sequence $(a_N)_{N\in \N^*}$ with $a_N = \mathcal{O}(N^{-a_0})$ as $N \to \infty$ for some $a_0 > 0$ such that
    $$\forall N \in \N^*: \quad \exists \Phi_N \in \mathcal{F}_N : \quad \mathcal{E}_\Omega(\Phi_N) \leq a_N,$$
    \item \label{enu:EntropyCondition} there exists $C > 0$ and $(\delta_N)_{N\in \N^*} \subseteq (0, \infty)$ such that
    $$\forall N \in \N^*: \quad V_{\mathcal{F}_N, \norm{.}_{L^\infty}} (\delta_N)\leq C \cdot N \cdot \delta_N$$ 
    \item \label{enu:TechnicalCondition} there exists $\iota > 0$ such that the sequence $(\eta_N)_{N\in \N^*} \subseteq \R$ with $\eta_N := \max\{ a_N, \delta_N\}$, $N \in \N^*$, satisfies 
    $$N^{1-\iota} \cdot \eta_N\gtrsim 1\qquad \forall N \in \N^*,$$
\end{enumerate}
then $$\mathbb{E}_S\lambda^n(\{ x\in Q_n:\chi_\Omega(x) \neq (\chi_{[1/2,\infty)} \circ \Phi_S) (x)\}) \lesssim  \eta_N,$$
where the constant only depends on the constants from Conditions \ref{enu:approxcond}, \ref{enu:EntropyCondition}, and \ref{enu:TechnicalCondition}.
\end{theorem}

It remains to verify conditions (i), (ii), and (iii) of Theorem \ref{thm:5.6optimal} to infer Theorem \ref{thm:statlearn}.

(i) Lemma \ref{lem:estim1} yields 
$$\mathcal{E}_{\Omega}(\Phi_\eps^\Omega) \leq 2\eps,$$
hence we set $a_N:= 2 N^{-m/(m+pn-p)}$ for $N$ large to get $a_N = \mathcal{O}(N^{-a_0})$ for $a_0 := m/(m+pn-p)>0$. For the initial segment of the sequence, take $\Phi_N\in\mathcal{F}_N$ with realization constant zero and let $a_N:= 2=2\lambda^n(Q_n)$.

(ii) This uses \cite{grohs2024proof}, Lemma 6.1, applied to the class of continuous functions $C(Q_n)$ instead of $C([0,1]^n)$: For each $0<\delta\leq 1$, the $\delta$-covering number of $\mathcal{F}_N$ is bounded by
$$(44/\delta \cdot \lfloor L+1\rfloor^4 \cdot(\lceil B_N \rceil\max\{n,\lfloor W_N\rfloor\})^{\lfloor L+2\rfloor})^{\lfloor W_N\rfloor},$$
so the covering entropy satisfies 
$$V_{\mathcal{F}_N,\norm{.}_{L^\infty}}(\delta)  \leq W_N\cdot(\log(44(L+1)^4)+\log(1/\delta)+(L+2)(1+\log(B_N))+(L+2)\log(\max \{n,W_N\})).$$
Hence for every $$0<\kappa<\min\{p(n-1)/(m+pn-p),m/(m+pn-p)\}$$
we get $0<\delta_N:= N^{\kappa-m/(m+pn-p)}\leq 1$ and estimate for $N$ sufficiently large
\begin{align*}
V_{\mathcal{F}_N,\norm{.}_{L^\infty}}(\delta_N)  & \leq cN^{p(n-1)/(m+pn-p)}(\log(44(L+1)^4) + m/(m+pn-p)\log N -\kappa\log N \\
& + (L+2)(1+sm/(m+pn-p)\log N) \\
& + (L+2)\max\{ \log n,\log(c) + p(n-1)/(m+pn-p) \log N\}) \\
& \leq \tilde{c} N^{p(n-1)/(m+pn-p)}\log N \leq \hat{c} N \delta_N
\end{align*}
for some constants $\tilde{c}=\tilde{c}(n,m,\ell,p,s,c,\kappa), \hat{c}=\hat{c}(n,m,\ell,p,s,c,\kappa) > 0$ independent of $N$ and $\Omega$. Adjusting the constant $\hat{c}$ establishes the claim for all $N$.

(iii) For sufficiently large $N$, the number $\eta_N:= \max\{a_N,\delta_N\} = \max\{2N^{-m/(m+pn-p)}, N^{\kappa-m/(m+pn-p)}\}$ in combination with 
$$\iota := p(n-1)/(m+pn-p)-\kappa > 0$$
satisfies $\eta_N \cdot N^{1-\iota} \gtrsim 1$. Adjusting the implicit constant establishes the claim for all $N$.

Hence for $N$ large enough, Theorem \ref{thm:5.6optimal} yields $$\mathbb{E}_S\lambda^n(\{ x\in Q_n:\chi_\Omega(x) \neq \chi_{[1/2,\infty)} \circ \Phi_S (x)\}) \lesssim N^{\kappa-m/(m+pn-p)},$$
and enlarging the intrinsic constant yields the result for all $N$, proving Theorem \ref{thm:statlearn}.

\subsection*{Future work}

To conclude, we give a brief treatment of a possible extension of our approximation theorem. 
Parallel to Theorem \ref{thm:mainapprox}, it also makes sense to consider Lipschitz extensions of the signature datum instead of assuming extendability of traceable sets. 
At the cost of losing regularity $m$, it yields a broader approximation result applicable to \emph{all} definable subsets of $Q_n$:

Let $A\subseteq Q_n$ be a definable set. By \cite{pawlucki2008lipschitz}, Theorem $3_n$, $\interior(A)$ can be partitioned into finitely many sets $S^1,\ldots,S^k, D \subseteq Q_n$ where $D$ is nowhere dense and for each $j\in[k]$, there is a permutation of coordinates $P_j:\R^n\to\R^n$ so that $P_jS^j$ is $1$-traceable such that its signature datum has bounded partial derivatives.
Using that $\lambda^n(D)=\lambda^n(A\setminus \interior(A)) = 0$, we get that $\lambda^n(A\setminus \bigcup_{i=1}^k S^i)=0$.

For the analytical counterpart, we consider \emph{$B$-Lipschitz traceable sets} $A\subseteq Q_n$ for $B>0$ that are open traceable sets whose signature datum $(\{f_1,g_1\},\ldots,\{f_n,g_n\})$ has extensions $F_i,G_i:Q_{i-1}\to\R$ satisfying $\norm{F_i}_{L^\infty(Q_{i-1})},\norm{G_i}_{L^\infty(Q_{i-1})} \leq B$ with Lipschitz constant bounded by $B$. The canonical extensions
$$F_i(x):=\inf_{x_0\in S_{i-1}^j} f_i^j(x_0) + B\norm{x-x_0}, \quad G_i(x):= \inf_{x_0\in S_{i-1}^j} g_i^j(x_0) + B\norm{x-x_0}$$
can be clipped at $-1/2$ and $1/2$ to form Lipschitz boundaries, which replaces the need for Whitney's extension theorem.
We introduce the class $\mathcal{L}_n^{B,\ell}$ of sets $\Omega \subseteq Q_n$ that satisfy
$$\lambda^n(\Omega \triangle \bigcup_{i=1}^l P_iA^i)=0$$
for some $B$-Lipschitz traceable sets $A^1,\ldots,A^l$ with $0\leq l \leq \ell$ and permutations of coordinates $P_1,\ldots,P_l\in\R^{n\times n}$.
By \cite{pawlucki2008lipschitz}, there is a map $B: \N\to\R$ such that every definable subset of $Q_n$ belongs to $\bigcup_{\ell\in\N}\mathcal{L}_n^{B(n),\ell}$.

Considering horizon functions $\mathrm{F}_i^j$ and $\mathrm{G}_i^j$ coming from the extensions $F_i^j,G_i^j$ clipped at $-1/2$ and $1/2$ of the signature datum of $A^j$, and observing that for large enough $B>0$ the class $\mathcal{HF}_{1,n,B}$ in \cite{petersen2018optimal} contains $\mathrm{F}_i^j$ and $1-\mathrm{G}_i^j$, we can follow the construction of approximating ReLU neural networks in Section \ref{sec:approx} with $m=1$. By pre-composing the first layer of the resulting ReLU neural network that approximates $\chi_{P_jA^j}$ with $P_j^{-1}$, we get that $\chi_{P_jA^j}$ can be approximated up to $L^p$ error $\eps$ with $\mathcal{O}(\eps^{-p(n-1)})$ many weights, all bounded by a polynomial in $1/\eps$ and layer count independent of $\eps$. The construction of a ReLU neural network approximating $\chi_\Omega$ is then analogous to the last part of Section \ref{sec:approx}.

\section*{Acknowledgements}

We want to thank Allen Gehret for useful discussions and helpful insights. P.P. was supported by the Austrian Science Fund (FWF) Project P-37010.

\subsection*{AI declaration}
GPT 6.1 Sol was used to generate the figures, revise the text for typos, and assisted in mathematical proof reading. Almost all content of the paper remains handwritten and the authors take full responsibility of this text.

\printbibliography

@misc{petersen2025mathematicaltheorydeeplearning,
      title={Mathematical theory of deep learning}, 
      author={Philipp Petersen and Jakob Zech},
      year={2026},
      eprint={2407.18384},
      archivePrefix={arXiv},
      primaryClass={cs.LG},
      url={https://arxiv.org/abs/2407.18384}, 
}

@article{petersen2018optimal,
  title={Optimal approximation of piecewise smooth functions using deep {R}e{LU} neural networks},
  author={Petersen, Philipp and Voigtlaender, Felix},
  journal={Neural Networks},
  volume={108},
  pages={296--330},
  year={2018},
  publisher={Elsevier}
}

@article{garcia2025high,
  title={High-dimensional classification problems with {Barron} regular boundaries under margin conditions},
  author={Garc{\'\i}a, Jonathan and Petersen, Philipp},
  journal={Neural Networks},
  volume={192},
  year={2025},
  eid={107898},
  publisher={Elsevier}
}

@article{caragea2023neural,
  title={Neural network approximation and estimation of classifiers with classification boundary in a {Barron} class},
  author={Caragea, Andrei and Petersen, Philipp and Voigtlaender, Felix},
  journal={The Annals of Applied Probability},
  volume={33},
  number={4},
  pages={3039--3079},
  year={2023},
  publisher={Institute of Mathematical Statistics}
}

@article{meyer2023optimal,
  title={Optimal convergence rates of deep neural networks in a classification setting},
  author={Meyer, Joseph T},
  journal={Electronic Journal of Statistics},
  volume={17},
  number={2},
  pages={3613--3659},
  year={2023},
  publisher={The Institute of Mathematical Statistics and the Bernoulli Society}
}

@article{petersen2021optimal,
  title={Optimal learning of high-dimensional classification problems using deep neural networks},
  author={Petersen, Philipp and Voigtlaender, Felix},
  journal={arXiv preprint arXiv:2112.12555},
  year={2021}
}

@article{kim2021fast,
  title={Fast convergence rates of deep neural networks for classification},
  author={Kim, Yongdai and Ohn, Ilsang and Kim, Dongha},
  journal={Neural Networks},
  volume={138},
  pages={179--197},
  year={2021},
  publisher={Elsevier}
}

@misc{yarotsky2017errorboundsapproximationsdeep,
      title={Error bounds for approximations with deep {R}e{LU} networks}, 
      author={Dmitry Yarotsky},
      year={2017},
      eprint={1610.01145},
      archivePrefix={arXiv},
      primaryClass={cs.LG},
      url={https://arxiv.org/abs/1610.01145}, 
}

@article{grohs2024proof,
  title={Proof of the theory-to-practice gap in deep learning via sampling complexity bounds for neural network approximation spaces},
  author={Grohs, Philipp and Voigtlaender, Felix},
  journal={Foundations of Computational Mathematics},
  volume={24},
  number={4},
  pages={1085--1143},
  year={2024},
  publisher={Springer}
}

@article{lecun2015deep,
  title={Deep learning},
  author={LeCun, Yann and Bengio, Yoshua and Hinton, Geoffrey},
  journal={Nature},
  volume={521},
  number={7553},
  pages={436--444},
  year={2015},
  publisher={Nature Publishing Group UK London}
}

@article{bareilles2025deep,
  title={Deep Learning as the Disciplined Construction of Tame Objects},
  author={Bareilles, Gilles and Gehret, Allen and Aspman, Johannes and Lep{\v{s}}ov{\'a}, Jana and Mare{\v{c}}ek, Jakub},
  journal={arXiv preprint arXiv:2509.18025},
  year={2025}
}

@article{ma2022barron,
  title={The {Barron} space and the flow-induced function spaces for neural network models},
  author = {E, Weinan and Ma, Chao and Wu, Lei},
  journal={Constructive Approximation},
  volume={55},
  number={1},
  pages={369--406},
  year={2022},
  publisher={Springer}
}

@book{van1998tame,
  title={Tame Topology and O-minimal structures},
  series={London Mathematical Society Lecture Note Series},
  author={van den Dries, Lou},
  volume={248},
  year={1998},
  publisher={Cambridge University Press}
}

@book{coste1999introduction,
  title={An introduction to o-minimal geometry},
  author={Coste, Michel},
  year={1999}
}

@article{wilkie1999theorem,
  title={A theorem of the complement and some new o-minimal structures},
  author={Wilkie, Alex J},
  journal={Selecta Mathematica},
  volume={5},
  number={4},
  pages={397--421},
  year={1999},
  publisher={Springer}
}

@article{rolin2003quasianalytic,
  title={Quasianalytic Denjoy-Carleman classes and o-minimality},
  author={Rolin, J-P and Speissegger, Patrick and Wilkie, Alex},
  journal={Journal of the American Mathematical Society},
  volume={16},
  number={4},
  pages={751--777},
  year={2003}
}

@article{van1994elementary,
  title={The elementary theory of restricted analytic fields with exponentiation},
  author={van den Dries, Lou and Macintyre, Angus and Marker, David},
  journal={Annals of Mathematics},
  volume={140},
  number={1},
  pages={183--205},
  year={1994},
}

@incollection{miller2012basics,
  title={Basics of o-minimality and Hardy fields},
  author={Miller, Chris},
  booktitle={Lecture notes on o-minimal structures and real analytic geometry},
  pages={43--69},
  year={2012},
  publisher={Springer}
}

@article{pawlucki2008lipschitz,
  title={Lipschitz cell decomposition in o-minimal structures I},
  author={Paw{\l}ucki, Wies{\l}aw},
  journal={Illinois Journal of Mathematics},
  volume={52},
  number={3},
  pages={1045--1063},
  year={2008},
  publisher={Duke University Press}
}

@article{aschenbrenner2019whitney,
  title={Whitney’s extension problem in o-minimal structures},
  author={Aschenbrenner, Matthias and Thamrongthanyalak, Athipat},
  journal={Revista matem{\'a}tica iberoamericana},
  volume={35},
  number={4},
  pages={1027--1052},
  year={2019}
}

@article{le1998verdier,
  title={Verdier and strict Thom stratifications in o-minimal structures},
  author={L{\^e} Loi, Ta},
  journal={Illinois Journal of Mathematics},
  volume={42},
  number={2},
  pages={347--356},
  year={1998},
  publisher={Duke University Press}
}

@article{fischer2007minimal,
  title={O-minimal {$\Lambda^m$}-regular stratification},
  author={Fischer, Andreas},
  journal={Annals of Pure and Applied Logic},
  volume={147},
  number={1-2},
  pages={101--112},
  year={2007},
  publisher={Elsevier}
}

@article{karpinski1997polynomial,
  title={Polynomial bounds for {VC} dimension of sigmoidal and general {Pfaffian} neural networks},
  author={Karpinski, Marek and Macintyre, Angus},
  journal={Journal of Computer and System Sciences},
  volume={54},
  number={1},
  pages={169--176},
  year={1997},
  publisher={Elsevier}
}

@article{chase2019model,
  title={Model theory and machine learning},
  author={Chase, Hunter and Freitag, James},
  journal={Bulletin of Symbolic Logic},
  volume={25},
  number={3},
  pages={319--332},
  year={2019},
  publisher={Cambridge University Press}
}

@article{laskowski1992vapnik,
  title={Vapnik-Chervonenkis classes of definable sets},
  author={Laskowski, Michael C.},
  journal={Journal of the London Mathematical Society},
  volume={s2-45},
  number={2},
  pages={377--384},
  year={1992},
  publisher={London Mathematical Society}
}

@book{aschenbrenner2017asymptotic,
  title={Asymptotic Differential Algebra and Model Theory of Transseries},
  author={Aschenbrenner, Matthias and van den Dries, Lou and van der Hoeven, Joris},
  year={2017},
  series = {Annals of Mathematics Studies},
  number = {195},
  publisher={Princeton University Press}
}

@misc{krapp2025measurabilityfundamentaltheoremstatistical,
      title={Measurability in the Fundamental Theorem of Statistical Learning}, 
      author={Lothar Sebastian Krapp and Laura Wirth},
      year={2025},
      eprint={2410.10243},
      archivePrefix={arXiv},
      primaryClass={cs.LG},
      url={https://arxiv.org/abs/2410.10243}, 
}

@misc{kratsios2026feedforwardneuralnetworkdefinable,
      title={Every Feedforward Neural Network Definable in an o-Minimal Structure Has Finite Sample Complexity}, 
      author={Anastasis Kratsios and Gregory Cousins and Haitz Sáez de Ocáriz Borde and Bum Jun Kim and Simone Brugiapaglia},
      year={2026},
      eprint={2605.07097},
      archivePrefix={arXiv},
      primaryClass={stat.ML},
      url={https://arxiv.org/abs/2605.07097}, 
}

@book{rockafellar1998variational,
  title={Variational analysis},
  author={Rockafellar, Ralph Tyrrell and Wets, Roger JB},
  year={1998},
  publisher={Springer}
}

@article{whitney1934analytic,
  title={Analytic extensions of differentiable functions defined in closed sets},
  author={Whitney, Hassler},
  journal={Transactions of the American Mathematical Society},
  volume={36},
  number={1},
  pages={63--89},
  year={1934}
}

@article{bareilles2026piecewise,
  title={Piecewise polynomial regression of tame functions via integer programming},
  author={Bareilles, Gilles and Aspman, Johannes and N{\v{e}}me{\v{c}}ek, Ji{\v{r}}{\'\i} and Mare{\v{c}}ek, Jakub},
  journal={Computer Aided Geometric Design},
  pages={102617},
  year={2026},
  publisher={Elsevier}
}

\end{document}